\documentclass[reqno]{amsart}
\usepackage{url}
\usepackage{amssymb,amsthm,amsfonts,amstext,amsmath}
\usepackage{newcent}      
\usepackage{helvet}       
\usepackage{courier}     
\usepackage{graphicx}
\usepackage{enumerate}
\usepackage{hyperref}
\numberwithin{equation}{section}
\usepackage{color}
\usepackage{float}
\usepackage{mathrsfs}
\usepackage{tikz}
\usepackage{float}
\usetikzlibrary{calc}
\definecolor{light-gray}{gray}{0.95}
\usepackage{float}
\newtheorem{theorem}{Theorem}[section]

\newtheorem{lemma}[theorem]{Lemma}
\newtheorem{proposition}[theorem]{Proposition}

\newtheorem{remark}[theorem]{Remark}
\newtheorem{definition}{Definition}
\newtheorem{open}{Open Question}

\newcommand{\mc}[1]{{\mathcal #1}}
\newcommand{\mf}[1]{{\mathfrak #1}}
\newcommand{\mb}[1]{{\mathbf #1}}
\newcommand{\bb}[1]{{\mathbb #1}}

\newcommand{\eps}{\varepsilon}

\newcommand{\Z}{\mathbb{Z}}

\newcommand{\<}{\langle}
\renewcommand{\>}{\rangle}
\newcommand{\p}{\partial}
\newcommand{\pfrac}[2]{\genfrac{}{}{}{1}{#1}{#2}}

\newcommand\topo[2]{\genfrac{}{}{0pt}{}{#1}{#2}}
\newcommand{\Wd}{\partial W_{\rm\text{diag}}}
\usetikzlibrary{positioning}
\newcommand*\mytriangle{%
\hspace{0.8pt}
\begin{tikzpicture}[line join= round]
\begin{scope}
\draw (0,0)--(0,5pt)--(5pt,5pt)--cycle;
\end{scope}
\end{tikzpicture}}
\newcommand{\mca}{\mathcal{A}}

\newcommand{\A}{\Delta_\alpha}
\newcommand{\B}{\nabla_\alpha}
\newcommand{\Y}{\mathcal{Y}}

\def\centerarc[#1](#2)(#3:#4:#5){\draw[#1] ($(#2)+({#5*cos(#3)},{#5*sin(#3)})$) arc (#3:#4:#5);}

\let\oldtocsection=\tocsection
\let\oldtocsubsection=\tocsubsection
\let\oldtocsubsubsection=\tocsubsubsection
\renewcommand{\tocsection}[2]{\hspace{0em}\oldtocsection{#1}{#2}}
\renewcommand{\tocsubsection}[2]{\hspace{1em}\oldtocsubsection{#1}{#2}}
\renewcommand{\tocsubsubsection}[2]{\hspace{2em}\oldtocsubsubsection{#1}{#2}}
\DeclareRobustCommand{\SkipTocEntry}[5]{}
\begin{document}
\title[Non-equilibrium fluctuations for the SSEP with a slow bond]{Non-equilibrium fluctuations\\ for the SSEP with a slow bond}

\author[D. Erhard]{D. Erhard}
\address{UFBA\\
 Instituto de Matem\'atica, Campus de Ondina, Av. Adhemar de Barros, S/N. CEP 40170-110\\
Salvador, Brazil}
\curraddr{}
\email{erharddirk@gmail.com}
\thanks{}

\author[T. Franco]{T. Franco}
\address{UFBA\\
 Instituto de Matem\'atica, Campus de Ondina, Av. Adhemar de Barros, S/N. CEP 40170-110\\
Salvador, Brazil}
\curraddr{}
\email{tertu@ufba.br}
\thanks{}

\author{P. Gon\c{c}alves}
\address{Center for Mathematical Analysis,  Geometry and Dynamical Systems,
Instituto Superior T\'ecnico, Universidade de Lisboa,
Av. Rovisco Pais, 1049-001 Lisboa, Portugal.}
\curraddr{}
\email{patricia.goncalves@math.tecnico.ulisboa.pt}
\thanks{}

\author{A. Neumann}
\address{UFRGS, Instituto de Matem\'atica e Estat\'istica, Campus do Vale,\newline Av. Bento Gon\c calves, 9500. CEP 91509-900, Porto Alegre, Brasil}
\curraddr{}
\email{aneumann@mat.ufrgs.br}
\thanks{}

\author[M. Tavares]{M. Tavares}
\address{UFBA\\
 Instituto de Matem\'atica, Campus de Ondina, Av. Adhemar de Barros, S/N. CEP 40170-110\\
Salvador, Brazil}
\curraddr{}
\email{tavaresaguiar57@gmail.com}
\thanks{}


\begin{abstract} 
We prove the non-equilibrium fluctuations for the one-dimens-ional symmetric simple exclusion process with a slow bond. This generalizes a result of \cite{fgn2,fgn2corrigendum}, which dealt with the equilibrium fluctuations. The foundation stone of our proof is a precise estimate on the  correlations of the system, and that is by itself one of the main novelties of this paper. To obtain these estimates, we first deduce a spatially discrete PDE for the covariance function and we relate it to the local times of a random walk in a non-homogeneous environment via Duhamel's principle. Projection techniques and coupling arguments reduce the analysis to the problem of studying the local times of the classical random walk. We think that the method developed here can be applied to a variety of models, and we provide a discussion on this matter.\vspace{-1cm}
\end{abstract}


\renewcommand{\subjclassname}{\textup{2010} Mathematics Subject Classification}

\maketitle
\tableofcontents

\section{Introduction}\label{s1}

One of the most challenging problems in the field of interacting particle systems is the derivation of the non-equilibrium fluctuations around the hydrodynamic limit and up to now there is not a satisfactory and robust theory that one can apply successfully. The main difficulty that one faces is to understand the precise  asymptotic behaviour of the long range correlations of the system. To be more precise, when letting the interacting system  start from a general measure (typically a non invariant measure for which the hydrodynamic limit can be obtained),   the correlations between any two sites are not null, but decay to zero as the scaling parameter $n$ grows. 
 
 In many situations a uniform bound on the  correlation function of order $O(1/n)$ is sufficient to obtain the non-equilibrium fluctuations of the system (see \cite{MasiSurvey,Ravi1992} for instance). For the model that we are going to describe in the sequel, the uniform bound on the correlation function happens to be of order $O(\log n/n)$, demanding new efforts both on the derivation of such a bound and on the application of such a bound on the proof of  the non-equilibrium fluctuations.

To be more specific, here we study  the symmetric simple exclusion process (SSEP) evolving on $\bb Z$ when a slow bond is added to it. The dynamics of this model is defined as follows. On $\bb Z$, particles at the vertexes of the  bond $\{x,x+1\}$ exchange positions at rate $1$, except at the particular bond $\{0,1\}$, where the rate of exchange is given by $\alpha/n$, with $\alpha\in(0,+\infty)$. Since the rate at the bond $\{0,1\}$ is slower with respect to the rates at other bonds, the bond $\{0,1\}$ coined the name \emph{slow bond}. Particles move on the one-dimensional lattice according to those rates of exchange and they are not created nor annihilated, being the spatial disposition of particles  the object of interest. 

The investigation on the behaviour of this process was initiated in \cite{fgn1} where the hydrodynamic limit was derived (see also \cite{Franco2010,fgn3}). By this we mean that the density of particles of the system converges to   a function $\rho_t(\cdot)$ which is a weak solution to a partial differential equation, called the \emph{hydrodynamic equation}. For the choice of the rates given above, the corresponding hydrodynamic equation is the one-dimensional heat equation with a boundary condition of Robin type:
\begin{equation}\label{robin_0}
\left\{
\begin{array}{ll}
 \partial_t \rho(t,u)\; = \; \partial^2_{uu} \rho(t,u)\,,\quad \text{ for }u\neq 0, \\
\p_u \rho (t,0^+)\;=\;\p_u \rho(t,0^-)\;=\;\alpha\big[\rho(t,0^+)-\rho(t,0^-)\big] , \\
 \rho(0,u) \;= \; \rho_0(u),
\end{array}
\right.
\end{equation} 
where $0^+$ and $0^-$ denote the side limits at zero from the right and from the left, respectively.

In fact, in \cite{fgn1} a more general choice for the rates was considered, and three different hydrodynamical behaviours were obtained. There, the slow bond was taken as the bond $\{-1,0\}$ instead of $\{0,1\}$, and the rate of exchange at that bond was given by $\frac{\alpha}{n^\beta}$, with $\beta\geq 0$ and $\alpha$ as given above. The choice of the slow bond as $\{0,1\}$ or $\{-1,0\}$ is essentially a matter of notation, having no special relevance. On the other hand, depending on the range of $\beta$, the boundary conditions of the hydrodynamic equation can be  of Neumann type (when $\beta>1$), which corresponds to \eqref{robin_0} with $\alpha=0$; or there is an absence of boundary conditions (when $\beta\in[0,1))$. The model  we approach  here corresponds to the choice $\beta=1$ in \cite{fgn1}. 

The effect  of the slow bond at a microscopic level is obvious: it narrows down the passage of particles across it. At a macroscopic level, its presence leads to boundary conditions in  the partial differential equation. By looking at the hydrodynamic equation \eqref{robin_0}, we see that the boundary conditions characterize the current of the system through the macroscopic position $u=0$. The boundary conditions state that the current is proportional to the difference of  concentration of the intervals $(0,+\infty)$ and $(-\infty,0)$ near the boundary, which is in agreement with \emph{Fick's Law}.

The \textit{equilibrium fluctuations} for this model were presented in \cite{fgn2} and three different Ornstein-Uhlenbeck processes were obtained, which again had the corresponding boundary conditions as seen at the hydrodynamical level. 
We extend here the results of \cite{fgn2}  by allowing the system to start from any measure and not necessarily from the stationary measure, namely the Bernoulli product measure, as required in \cite{fgn2}. 
The choice of  rates as described above is restricted to $\beta=1$ so that we are in the Robin's regime. 

As the main theorem, we prove the non-equilibrium fluctuations and show that they are given by an \textit{Ornstein-Uhlenbeck process with Robin boundary conditions}.  By an Ornstein-Uhlenbeck process with  Robin boundary conditions it should be understood, in the same spirit as in \cite{fgn2}, that these boundary conditions are encoded in the space of test functions, see  \eqref{eq:bound_beta_1} below. Microscopically, the role of the boundary conditions at the level of the test functions is to force some additive functionals that appear in the Dynkin martingale to vanish as $n$ grows. If we do not impose the boundary conditions of \eqref{eq:bound_beta_1} on the test functions, then we would need some extra arguments to control those additive functionals. This is left to a future work.

The proof's structure is the standard one in the theory of stochastic processes:  tightness for the sequence of density fields together with uniqueness of  limit points. 
Let us discuss next the features of the work, besides the non-equilibrium result itself. And at same time we give the outline of the paper.

The biggest difficulty we face in our proof  is undoubtedly the fact that the  slow bond decreases the speed at which  correlations vanish. In the usual SSEP, where  all bonds have rate one,  correlations are of order $O(1/n)$. In our case however,  correlations are of order $O(\log n / n)$, therefore  bigger than in the usual SSEP.  For sites on the same side of the slow bond this fact is  intuitive:  correlations should actually increase since it is more difficult for particles to cross the slow bond. Curiously, our proof shows that the same happens for sites at different sides of the slow bond, that is, correlations are of order $O(\log n/n)$ on the entire line. An intuition of why this happens is given in Remark~\ref{rm28}, and a discussion of why the bound $O(\log n / n)$ is sharp is made in Subsection~\ref{rm27}.

In Section~\ref{s2} we define the symmetric simple exclusion process in the presence of a slow bond at $\{0,1\}$, we introduce notations and we state the main results of the article. At the end of this section, three related open problems  are presented.

In Section~\ref{s4} we  establish connections between the two-point correlation function and the discrete derivative with the expected occupation time of a site of two-dimensional and one-dimensional random walks, respectively, in an inhomogeneous medium. This is one of the features: the way itself to estimate correlations via local times of random walks, which we believe may be applied to different contexts. The idea behind that is actually simple. We  express both the discrete derivative and the  correlation function as  solutions to some discrete equations, then we use Duhamel's Principle to write each one of these solutions in terms of transition probabilities of random walks, in 1-d when looking at the discrete derivative and in 2-d when looking at the correlation function. Then,  the local times of these random walks show up naturally from these arguments and we need to establish optimal bounds for them.

Since the necessary estimates for local times of random walks were not yet available in the literature, we derive them in Section~\ref{s5} by means of projection of Markov chains (also known as \textit{lumping}) and couplings. The statements of those estimates may look artificial at first glance, but they naturally appear when one looks for estimates on the discrete derivative of the occupation average at a site and for the two-point correlation function, as aforementioned.

An additional feature  is about uniqueness of  the Ornstein-Uhlenbeck process with  Robin boundary conditions in the non-equilibrium setting, where the variance is governed by the PDE \eqref{robin_0}. Suitably  adapting the proofs of \cite{HolleyStroock,kl}, we give a slightly more general version of uniqueness, which permits to consider more general starting measures than the usual slowly varying product measure. The generalization here consists on supposing that the density field associated to the initial measure does not necessarily converge to a \textit{Gaussian field}, but only to \textit{some field}.
Moreover, this proof of uniqueness has a pedagogical importance, since the original proof of uniqueness for the Ornstein-Uhlenbeck process in the non-equilibrium setting, to the best of our knowledge, is not available in the literature.

Finally, in Section~\ref{s3} we  present the proof of the density fluctuations, which  relies  on the estimates  of the  discrete derivative of expected occupation number at a site, and on the two-point correlation function. A small but important detail is  the fact that the estimate on the discrete derivative is sufficient for our purposes. In previous works (\cite{MasiSurvey,Ravi1992}), the proof of non-equilibrium fluctuations was based on the convergence of the spatially discretized heat equation towards the continuum heat equation.  Such an approximation is quite good, of order $O(n^{-2})$, and quite hard to adapt to the non-homogeneous medium set up without some uniform ellipticity assumption as in \cite{jara2008}.   On the other hand, the discrete derivative estimate for the spatially discretized PDE is much easier to reach, as seen here. This idea on making use of the discrete derivative first appeared in \cite{FGN2018}, but its utility  becomes more evident now.

\section{Statement of results}\label{s2}

\subsection{The model}
We fix a parameter $\alpha>0$, and we consider the symmetric simple exclusion process $\{\eta_t:\,t\geq{0}\}$ with a slow bound as  defined in \cite{fgn1}. More precisely, $\{\eta_t:\,t\geq{0}\}$ is the Markov process with state space $\Omega\overset{\text{def}}{=}\{0,1\}^{\bb Z}$, and infinitesimal generator $\mc L_{n}$ acting on local functions $f:\Omega\rightarrow \bb{R}$ via 
\begin{equation}\label{ln}
\begin{split}
(\mc L_n f)(\eta)\;=\;
\sum_{x\in \bb Z}\xi_{x,x+1}^n\Big(f(\eta^{x,x+1})-f(\eta)\Big)
\end{split}
\end{equation}
where 
\begin{equation}\label{xi}
\xi_{x,x+1}^n \;\overset{\text{def}}{=}\; 
\begin{cases}
 1\,, & \textrm{ if } \; x\neq 0\,,\\
\frac{\alpha}{n}\,, & \textrm{ if }\; x=0\,.\\ 
 \end{cases} 
\end{equation}
Here, for any $x\in\bb Z$, the configuration $\eta^{x,x+1}$ is obtained from $\eta$ by exchanging the occupation variables $\eta(x)$ and $\eta(x+1)$, i.e.,
\begin{equation*}
(\eta^{x,x+1})(y)\;=\;\left\{\begin{array}{cl}
\eta(x+1)\,,& \mbox{if}\,\,\, y=x\,,\\
\eta(x)\,,& \mbox{if} \,\,\,y=x+1\,,\\
\eta(y)\,,& \mbox{otherwise,}
\end{array}
\right.
\end{equation*}
see Figure~\ref{fig1} for an illustration of the jump rates.
Given $\eta\in \{0,1\}^\bb Z$, we then say that the site $x\in\bb Z$ is vacant if $\eta(x)=0$ and  occupied if $\eta(x)=1$.
\begin{figure}[!htb]
\centering
\begin{tikzpicture}
\centerarc[thick,<-](1.5,0.3)(10:170:0.45);
\centerarc[thick,->](1.5,-0.3)(-10:-170:0.45);
\centerarc[thick,<-](2.5,0.3)(10:170:0.45);
\centerarc[thick,->](2.5,-0.3)(-10:-170:0.45);
\centerarc[thick,->](3.5,-0.3)(-10:-170:0.45);
\centerarc[thick,<-](3.5,0.3)(10:170:0.45);
\centerarc[thick,->](4.5,-0.3)(-10:-170:0.45);
\centerarc[thick,<-](4.5,0.3)(10:170:0.45);
\centerarc[thick,->](5.5,-0.3)(-10:-170:0.45);
\centerarc[thick,<-](5.5,0.3)(10:170:0.45);
\draw (-1,0) -- (8,0);

\shade[ball color=black](0,0) circle (0.25);
\shade[ball color=black](1,0) circle (0.25);
\shade[ball color=black](3,0) circle (0.25);
\shade[ball color=black](6,0) circle (0.25);


\filldraw[fill=white, draw=black]
(2,0) circle (.25)
(4,0) circle (.25)
(5,0) circle (.25)
(7,0) circle (.25)
;

\draw (0.3,-0.05) node[anchor=north] {\small  $\bf -\!3$};
\draw (1.3,-0.05) node[anchor=north] {\small $\bf - 2 $};
\draw (2.3,-0.05) node[anchor=north] {\small $\bf - 1 $};
\draw (3.3,-0.05) node[anchor=north] {\small $\bf 0$};
\draw (4.3,-0.05) node[anchor=north] {\small $\bf 1$};
\draw (5.3,-0.05) node[anchor=north] {\small $\bf 2$};
\draw (6.3,-0.05) node[anchor=north] {\small $\bf 3$};
\draw (7.3,-0.05) node[anchor=north] {\small $\bf 4$};
\draw (1.5,0.8) node[anchor=south]{$1$};
\draw (1.5,-0.8) node[anchor=north]{$1$};
\draw (2.5,0.8) node[anchor=south]{$1$};
\draw (2.5,-0.8) node[anchor=north]{$1$};
\draw (3.5,0.8) node[anchor=south]{$\alpha/n$};
\draw (4.5,-0.8) node[anchor=north]{$1$};
\draw (4.5,0.8) node[anchor=south]{$1$};
\draw (5.5,-0.8) node[anchor=north]{$1$};
\draw (5.5,0.8) node[anchor=south]{$1$};
\draw (3.5,-0.8) node[anchor=north]{$\alpha/n$};
\end{tikzpicture}
\caption{Jump rates. The bond $\{0,1\}$ has a particular jump rate associated to it, which is given by $\alpha/n$.}\label{fig1}
\end{figure}
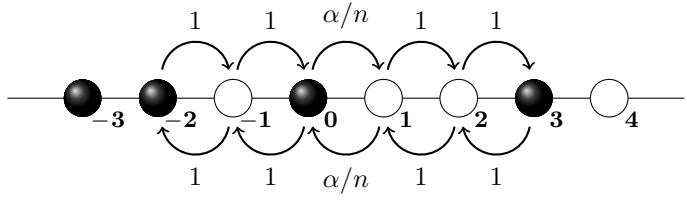

\subsection{Hydrodynamic limit}
Fix a  measurable density profile $ \rho_0: \bb R \rightarrow [0,1 ]$. For each  $n \in \bb N$, let $\mu_n$ be a probability measure on $\Omega$.  We say that the sequence $\{\mu_n\}_{n\in \bb N}$ is \textit{associated} with the profile $\rho_0(\cdot)$ if,  for any $ \delta >0 $ and any continuous function of compact support  $ f: \bb R\to\bb R $, the following holds:
\begin{equation}\label{eq3}
\lim_{n\to\infty}
\mu_{n} \Bigg[ \eta:\, \Big| \frac {1}{n} \sum_{x \in \bb Z} f(\pfrac{x}{n})\, \eta(x)
- \int f(u)\, \rho_0(u)\, du \Big| > \delta\Bigg] \;=\; 0\,.
\end{equation}
Fix $T>0$, and let  $\mc D([0,T],\Omega)$ be  the space of trajectories which are right continuous, with left limits and  taking values in $\Omega$.
Denote by $\bb P_{\mu_{n}} $ the probability measure on $\mc D([0,T],\Omega)$ induced by the SSEP with a slow bond accelerated by $n^2$, i.e., the Markov process with generator $n^2\mc L_n$, and initial measure $\mu_n$. With a slight abuse of notation, we also use the notation $\{\eta_t:t\in[0,T]\}$ for the accelerated process. Denote by $\bb E_{\mu_n}$ the expectation with respect to $\bb P_{\mu_n}$.
In~\cite{fgn1, fgn3} the  \emph{hydrodynamical behaviour} was studied. We note that the process there was studied in \textit{finite volume}, i.e., the model was considered on the discrete torus embedded into the continuous one-dimensional torus. However, since the extension to infinite volume is just a topological issue, the statement below can be obtained via an adaptation of the original approach:
\begin{theorem}[\cite{fgn1, fgn3}]\label{thm21}
Suppose that the 
sequence $\{\mu_n\}_{n\in \bb N}$ is associated to the profile $\rho_0(\cdot)$ in the sense of \eqref{eq3}. 
Then, for each $ t \in [0,T] $, for any $ \delta >0 $ and any continuous function $ f:\bb R\to\bb R $ with compact support, 
\begin{equation*}
\lim_{ n \rightarrow +\infty }
\bb P_{\mu_{n}} \Bigg[\,\eta_{\cdot} : \Big\vert \frac{1}{n} \sum_{x\in \bb Z}
f(\pfrac{x}{n})\, \eta_{t}(x) - \int_{\bb R} f(u)\, \rho(t,u)\, du\, \Big\vert
> \delta\, \Bigg] \;=\; 0\,,
\end{equation*}
 where  $\rho(t,\cdot)$ is the unique weak solution of the 
 heat equation with    Robin boundary conditions given by
\begin{equation}\label{robin}
\left\{
\begin{array}{ll}
 \partial_t \rho(t,u) = \; \partial^2_{uu} \rho(t,u), &t \geq 0,\, u \in \mathbb R\backslash\{0\},\\
\p_u \rho (t,0^+)=\p_u \rho(t,0^-)=\alpha\big[\rho(t,0^+)-\rho(t,0^-)\big],  &t \geq 0,\\
 \rho(0,u) = \; \rho_0(u), &u \in \mathbb R.
\end{array}
\right.
\end{equation}
\end{theorem}
Here, $\rho(t,0^+)$ and $\rho(t,0^-)$ denote the limit from the right and from the left at zero, respectively. The notation $0^\pm$ will be used throughout the article.

\subsection{Space of test functions and semigroup}
In this section we introduce a space of test functions, that is suitable for our purposes, and which, basically, coincides with the one  in~\cite{fgn2corrigendum}.   Here, functions  are continuous from the left at zero, while in~\cite{fgn2corrigendum} functions are continuous from the right. This subtle difference is due to choice of slow bond's position, which is $\{0,1\}$ here and $\{-1,0\}$ in~\cite{fgn2corrigendum}.

\begin{definition} We denote by $\mc S_\alpha(\bb R )$ the space of functions $f:\bb R\to\bb R$ such that: \begin{enumerate}
\item[(i)]  $f$ is smooth on $\bb R\backslash\{0\}$, i.e. $f \in C^{\infty}(\bb R\backslash\{0\})$,
\item[(ii)]  $f$ is continuous from the left at 0,
\item[(iii)]  for all non-negative integers $k,\ell$, the function $f$ satisfies
\begin{equation}\label{eq2.2}
\Vert f \Vert_{k,\ell}\;:=\; \sup_{u \neq 0} \Big| (1+|u|^\ell) \frac{d^k f}{du^k}(u)\Big| \;< \; \infty\,.
\end{equation}
\item[(iv)] for any integer $k\geq 0$,
\begin{equation} \frac{d^{2k+1} f}{du^{2k+1}}(0^+)\;= \;\frac{d^{2k+1}f}{du^{2k+1}}(0^-)\; =\; \alpha  \Bigg(\frac{d^{2k}f}{du^{2k}}(0^+)-\frac{d^{2k}f}{du^{2k}}(0^-)\Bigg)\,.
\label{eq:bound_beta_1}\end{equation}
\end{enumerate}
Moreover, $\mc S'_\alpha(\bb R)$ denotes the topological dual of $\mc S_\alpha(\bb R)$.
\end{definition}
In plain words, $\mc S_\alpha(\bb R)$ essentially consists of the space of functions in the Schwartz space $\mc S(\bb R)$ that are not necessarily smooth at the origin.
It is a consequence of~\eqref{eq2.2} that $\frac{d^k f}{du^k}(0^+)$ and $\frac{d^k f}{du^k}(0^-)$
exist for all integers $k \ge 0$. As in \cite{fgn2}, one may show that $S_\alpha(\mathbb{R})$ is a Fr\'echet space (this fact was only used when showing tightness, see \cite{Mitoma}). We recall below the explicit  formula for the semigroup that corresponds to the PDE~\eqref{robin}.
\begin{proposition}[\cite{fgn2}]\label{prop23}
Denote by $g_{\text{\rm even}}$ and $g_{\text{\rm odd}}$ the even and odd parts of a function $g:\bb R\to \bb R$, respectively. That is, for $u\in{\mathbb{R}}$,
\begin{equation*}
 g_{\text{\rm even}}(u)=\frac{g(u)+g(-u)}{2}\quad \text{and} \quad g_{\text{\rm odd}}(u)=\frac{g(u)-g(-u)}{2}\,.
\end{equation*}
 The solution of \eqref{robin} with initial condition $g\in\mc S_\alpha(\bb R)$ is given by
  \begin{equation*}
  \begin{split}
  & T_t^\alpha g(u)= \frac{1}{\sqrt{4\pi t}}\Bigg\{\int_{\bb R}
e^{-\frac{(u-y)^2}{4t}} g_{\textrm{{\rm even}}}(y)\,dy \\
    & + e^{2\alpha u}\!\!\int_u^{+\infty} \!\!\! e^{-2\alpha z} \!\!\int_0^{+\infty}\!
\Big[(\pfrac{z-y+4\alpha t}{2t})e^{-\frac{(z-y)^2}{4t}}+(\pfrac{z+y-4\alpha t}{2t})e^{-\frac{(z+y)^2}{4t}}\Big]\,
g_{\textrm{{\rm odd}}}(y)\, dy\, dz\,\Bigg\},\\
  \end{split}
  \end{equation*}
\noindent for $u>0$, and
  \begin{equation*}
  \begin{split}
 & T_t^{\alpha} g(u)= \frac{1}{\sqrt{4\pi t}}\Bigg\{\int_{\bb R}
e^{-\frac{(u-y)^2}{4t}} g_{\textrm{{\rm even}}}(y)\,dy \\
    & - e^{-2\alpha u}\!\!\int_{-u}^{+\infty}\!\!\! e^{-2\alpha z}\!\! \int_0^{+\infty}\!
\Big[(\pfrac{z-y+4\alpha t}{2t})e^{-\frac{(z-y)^2}{4t}}+(\pfrac{z+y-4\alpha t}{2t})e^{-\frac{(z+y)^2}{4t}}\Big]\,
g_{\textrm{{\rm odd}}}(y)\, dy\, dz\,\Bigg\},\\
  \end{split}
  \end{equation*}
\noindent for $u<0$.
\end{proposition}

The next proposition connects $T_t^\alpha$ with the space of test functions $\mc S_\alpha(\bb R)$.

\begin{proposition}[\cite{fgn2corrigendum}] The operator $T_t^\alpha$ 
defines a semigroup $T_t^\alpha:\mc S_\alpha(\bb R) \to\mc S_\alpha(\bb R)$. That is, for any given $g\in \mc S_\alpha(\bb R)$ and any time $t>0$, the solution $T_t^\alpha g$ of the PDE \eqref{robin} starting from $g$  also belongs to $\mc S_\alpha(\bb R)$.
\end{proposition}

\begin{definition}\label{def:laplacian_operator}
Let $\Delta_\alpha: \mc S_\alpha(\bb R)\rightarrow \mc S_\alpha(\bb R)$   be the Laplacian on $\mc S_\alpha(\bb R)$, i.e., for any $f\in \mc S_\alpha(\bb R)$, 
\begin{equation}\label{laplacian}
\Delta_\alpha f(u)\;=\;\left\{\begin{array}{cl}
\partial_{uu}^2 f(u)\,,& \mbox{if}\,\,\,u\neq0\,,\smallskip\\
\partial_{uu}^2 f(0^+)\,,& \mbox{if} \,\,\,u=0\,.\\
\end{array}
\right.
\end{equation}
The definition of the operator
$\nabla_\alpha: \mc S_\alpha(\bb R)\rightarrow C^{\infty}[0,1]$ is analogous.
\end{definition}

\subsection{Discrete derivatives and covariance estimatives}
Fix an initial measure $\mu_n$ on $\Omega$. For $x\in\bb Z$  and $t\geq 0$, let 
\begin{equation}\label{rho_t}
\rho^n_t(x)\;\overset{\text{def}}{=}\;\mathbb{E}_{\mu_n}\big[\eta_{t}(x)\big]\,.
\end{equation}
 A simple computation shows that $\rho_t^n(\cdot)$ is a solution of the discrete equation 
\begin{equation}\label{disc_heat}
 \partial_t \rho_t^n(x) \;= \; \big(n^2\mca_n \rho_t^n\big)(x)\,, \;\; x\in\bb Z\,,\;\;t \geq 0\,,
\end{equation}
 where the operator $\mca_n$ acts on functions $f:\bb Z\to\bb R$ as
\begin{equation}\label{op_B}
(\mca_nf)(x)\;:=\;\xi_{x,x+1}^n\Big(f(x+1)-f(x)\Big)+\xi_{x-1,x}^n\Big(f(x-1)-f(x)\Big)\,,~~\forall\,  x\in \bb Z\,,
\end{equation}
with $\xi_{x,x+1}$ as defined in \eqref{xi}.
\begin{definition}\label{defcorr}
For $x,y \in \bb Z$, and $t\in[0,T]$,  define the two-point correlation function
\begin{equation}
\varphi_t^n(x,y)\;\overset{\text{\rm def}}{=}\;\bb E_{\mu_n}\big[\eta_t(x)\eta_t(y)\big]-\rho_t^n(x)\rho_t^n(y)\,.
\end{equation}
\end{definition}
We now state two results that are fundamental for the study of density fluctuations, which are interesting by themselves.
\begin{theorem}[Discrete derivative estimate]\label{prop24}
Assume that there exists a constant $c > 0$ that does not depend on $n$ such that
\begin{equation}\label{assumptionone}
\sup_{ x\in \bb Z}|\rho_0^n (x)-\rho_0\big(\pfrac{x}{n}\big)|\;\leq\;\frac{c}{n}\,.
\end{equation}
Then, there exists $\overline{c}>0$ such that, for all $t\in[0,T]$, and all $n\in\bb N$, 
\begin{equation*}
\big|\rho_t^n(x+1)-\rho_t^n(x)\big|\;\leq\;
\begin{cases}
 \frac{\overline{c}}{n}, &\textrm{if }x\neq 0,\\
\overline{c}, &\textrm{if }x =0.
\end{cases}
\end{equation*}
\end{theorem}
Note that the second inequality above is obvious, but we kept in the statement of the theorem for the sake of clarity. 
\begin{theorem}[Correlation estimate]\label{Prop25}
Assume that there exists a constant $c > 0$ that does not depend on $n$ such that 
\begin{equation}\label{ass2}
\sup_{ (x,y)\in V}|\varphi_0^n (x,y)|\;\leq\;\frac{c}{n}\,.
\end{equation} 
Moreover, assume that~\eqref{assumptionone} is satisfied. Then, there exists $\hat{c}>0$ such that for all $n\in\bb N$,
\begin{equation}\label{eq:covarianceest}
\sup_{t\leq T}\sup_{(x,y)\in V} |\varphi_t^n(x,y)|\;\leq\; \frac{\hat{c}\log n}{n}\,,
\end{equation}
where $V:=\{(x,y)\in \bb Z\times \bb Z:\,y\geq x+1\}$.
\end{theorem}
\begin{remark}
Note that by the symmetry of the correlation function, Theorem~\ref{Prop25} immediately implies \eqref{eq:covarianceest} for $x\neq y$.
\end{remark}

\subsection{Ornstein-Uhlenbeck process}
Let $\rho(t,\cdot)$ be the unique  solution of the hydrodynamic equation \eqref{robin}.
In what follows, $\mathcal{D}([0,T],\mathcal{S}'_{\alpha}(\mathbb{R}))$
(resp.\break $\mathcal{C}([0,T],\mathcal{S}'_{\alpha}(\mathbb{R}))$)
denotes the space of  c\`adl\`ag (resp. continuous) $\mathcal{S}'_{\alpha}(\mathbb{R})$ valued functions endowed with the Skohorod topology. We also denote by $\chi$ the \textit{static compressibility} defined by $\chi(\rho)=\rho(1-\rho)$.
Denote by $\big\<\cdot,\cdot\>_{\rho_t(\cdot)}$ the inner product with respect to $L^2_{\Lambda_t}(\bb R)$, where the measure $\Lambda_t(du)$ is given by 
\begin{equation}\label{Lambda}
\Lambda_t(du)\;\overset{\text{def}}{=}\;2\chi(\rho_t(u)) du+ \frac{1}{\alpha}\Big[\rho_t(0^-)(1-\rho_t(0^+))+\rho_t(0^+)(1-\rho_t(0^-))\Big]\delta_0(du)\,, 
\end{equation}
where $\delta_0(du)$ denotes the Dirac measure at zero. More precisely, for $f,g\in\mc S_\alpha(\bb R)$,
\begin{align*}
\big\<f,g\big\>_{\rho_t(\cdot)}&\;=\; \int_{\bb R} 2\chi(\rho_t(u)) \,f (u) g(u)\,du\\
&\;+ \frac{1}{\alpha}\Big[\rho_t(0^-)(1-\rho_t(0^+))+\rho_t(0^+)(1-\rho_t(0^-))\Big]  f(0)g(0)\,.
\end{align*}
\begin{proposition}\label{prop26}
There exists a unique  (in distribution) random element  $\mc Y$  taking values in the space $\mc C([0,T],\mathcal{S}'_{\alpha}(\bb R))$ such
that the following two conditions hold:\medskip

\textbf{i)}  For every function $f \in \mathcal{S}_{\alpha}(\bb R)$, the stochastic processes $\mc M_t(f)$ and $\mc N_t(f)$ given by
\begin{align}
\mc M_t(f)&\;=\; \mc Y_t(f) -\mc Y_0(f) -  \int_0^t \mc Y_s(\A f)ds\,,\label{eq2.13}\\
\mc N_t(f)&\;=\;\big(\mc M_t(f)\big)^2 - \int_0^t \| \nabla_\alpha f\|_{\rho_s(\cdot)}^2\,ds\label{eq2.14}
\end{align}
are $\mc F_t$-martingales, where for each $t\in{[0,T]}$, $\mc F_t:=\sigma(\mc Y_s(f); s\leq t,  f \in \mathcal{S}_{\alpha}(\bb R))$.\medskip

 \textbf{ii)} $\mc Y_0$ is a random element taking values in $\mc S'_\alpha(\bb R)$ with a fixed distribution.\medskip

Moreover, if \textbf{i)} and \textbf{ii)} hold, then:\medskip

$\bullet$  for each $f\in\mc S_\alpha(\bb R)$,   conditionally to
$\mc F_s$ with  $s<t$, the 
distribution of $\Y_t(f)$ is normal  of mean $\Y_s(T^\alpha_{t-s}f)$ and variance $\int_s^{t}\Vert \B T^\alpha_{r}
f\Vert^2_{\rho_r(\cdot)}\,dr$.\medskip

$\bullet$ If $\mc Y_0$ is a Gaussian field, then the stochastic process $\{\Y_t(f)\,;\,t\geq 0\}$ will be Gaussian indeed.
\end{proposition}

In other words, if  $\mc Y^1$ and $\mc Y^2$ are two random elements taking values on  $\mc C([0,T],\mathcal{S}'_{\alpha}(\bb R))$ and satisfying the martingale problem described above by \textbf{i)} and \textbf{ii)}, then $\mc Y^1$ and $\mc Y^2$ must have the same distribution. 

It is common in the literature to write the martingale problem stated above as a \textit{formal} solution of some generalized stochastic partial differential equation. We discuss it with no mathematical rigor, aiming only at giving some intuition on the fluctuations' global behavior.
    
   We call the  random element $\mc Y$ defined via Proposition~\ref{prop26} a generalized Ornstein-Uhlenbeck  process which takes values on $\mathcal{C}([0,T],\mathcal{S}'_{\alpha}(\mathbb{R}))$ and it is the \textit{formal} solution of
\begin{equation}\label{eq Ou}
d\mathcal{Y}_t\;=\;\A \mathcal{Y}_tdt+\nabla_\alpha\, d\mc{W}_t\,,
\end{equation}
 where:
 
 $\bullet$ The operators $\A$ and $\nabla_\alpha$  have been given in Definition~\ref{def:laplacian_operator} and are usually referred to as the \textit{characteristics} of the Ornstein-Uhlenbeck process.
 
 $\bullet$ $\mc{W}$ is a space-time white noise with respect to the measure $\Lambda_s(du)$, i.e., $\mc W$ is a mean-zero Gaussian random element taking values in the dual space of $L^2_{\Lambda}([0,\infty)\times \bb R)$ with covariances given by
 \begin{align*}
 \bb E\Big[\mc W(F)\mc W(G)\Big]\;=\;\int_0^\infty \int_{\bb R} F(s,u)G(s,u)\,d\Lambda(s,u)\,,\quad\forall\, F,G\in L^2_{\Lambda}([0,\infty)\times \bb R)\,,
\end{align*}   
where $d\Lambda(s,u)=d\Lambda_s(u)\times ds$, and $\Lambda_s$ has been defined in \eqref{Lambda}. 
 
 $\bullet$ For $f\in \mc S_\alpha(\bb R)$, we define $\mc{W}_t(f):=\mc W(f\mb{1}_{[0,t]})$. 
 In particular, $\{\mc W_t(f): f\in \mc S_\alpha(\bb R)\}$ is a  Gaussian process with covariance given on $f,g\in\mc S_\alpha(\bb R)$ by
 \begin{align*}
 \bb E\Big[\mc W_t(f) \mc W_t(g)\Big]\;=\; \int_0^t\<f,g\>_{\rho_s(\cdot)}ds\,.
 \end{align*}

\subsection{Non-equilibrium fluctuations}
We define the density fluctuation field $\mc Y^n$ as the time-trajectory of a linear functional acting on functions $f\in\mc S_\alpha(\bb R)$ via
\begin{equation}\label{density field}
\mc Y^n_t(f)\;\overset{\text{def}}{=}\;\frac{1}{\sqrt{n}}\sum_{x\in\bb Z}f(\tfrac{x}{n})\Big(\eta_{t}(x)-\rho^n_t(x)\Big)\,.
\end{equation}
 For each $n\geq 1$, let  $Q_n$ be the probability measure on $\mc D([0,T],\mc S'_\alpha(\bb R))$  induced by the density fluctuation field $\mc Y^n$ and  a measure $\mu_n$. We now state the main result of this paper:
  
 \begin{theorem}[Non-equilibrium fluctuations]\label{thm26}
Consider the Markov processes $\{\eta_{t}: t\geq{0}\}$ starting from a  sequence of probability measures $\{\mu_n\}_{n\in \bb N}$ associated with a profile as in \eqref{eq3},  and assume:
\begin{enumerate}[\bf (A)]
\item Conditions \eqref{assumptionone} and \eqref{ass2} on mean and covariance, respectively.
\item There exists a $\mc S'_\alpha(\bb R)$-valued random variable $\mc Y_0$ such that $\mc Y^n_0$ converges in distribution to $\mc Y_0$, whose  law we denote  by $\mf L$.
\end{enumerate}
 Then, the sequence of processes $\{\mathcal{Y}_{t}^n\}_{ n\in{\bb N}}$ converges in distribution, as $n\rightarrow{+\infty}$, with respect to the
Skorohod topology
of $\mathcal{D}([0,T],\mathcal{S}'_{\alpha}(\bb R))$ to  a random element  $\mathcal{Y}$
in $\mathcal{C}([0,T],\mathcal{S}'_{\alpha}(\bb R))$, the generalized Ornstein-Uhlenbeck which is a solution of \eqref{eq Ou},  and $\mc Y_0$ has law $\mf L$.
\end{theorem} 

It is of worth to give examples of sequences $\{\mu_n\}_{n\in \bb N}$ of initial measures satisfying  assumptions \textbf{(A)} and \textbf{(B)}. Next, we present two examples of such initial measures  and we leave an open question on the subject.

The first example we present is the standard one for non-equilibrium fluctuations: take $\{\mu_n\}_{n\in \bb N}$ as the slowly varying Bernoulli product measure\break  $\{\nu^n_{\rho_0(\cdot)}\}_{n\in \bb N}$ associated with a smooth profile $\rho_0:\bb R\to[0,1]$, that is,  $\nu^n_{\rho_0(\cdot)}$ is a product measure on $\{0,1\}^{\bb Z}$ such that
\begin{align*}
\nu^n_{\rho_0(\cdot)}\big\{\eta\in \{0,1\}^{\bb Z}\,:\, \eta(x)=1\big\}\;=\; \rho_0\big(\pfrac{x}{n})\,.
\end{align*}
Obviously, \textbf{(A)} is satisfied. The proof that \textbf{(B)} holds is just an adaptation of the analogous result for the SSEP, being included in Proposition~\ref{convergence at time zero}  for the sake of completeness.

The second example we discuss is somewhat artificial, but, in any case,  illustrates the existence of a sequence of non-product measures satisfying \textbf{(A)} and \textbf{(B)}. Let $\mu_{n}$ be the measure on $\Omega$ induced by the distribution at the time $rn^2$, where $r>0$ is fixed,  of the (homogeneous) one-dimensional SSEP 
started from the slowly varying measure $\nu^n_{\rho_0(\cdot)}$ defined above.

From the \textit{propagation of local equilibrium} for the SSEP (see \cite{kl} and references therein), one can check that condition \eqref{assumptionone} holds. Besides that, it is well known that the SSEP has longe range correlations of order $O(1/n)$, giving \eqref{ass2}. Thus, assumption \textbf{(A)} is satisfied.
From the non equilibrium fluctuations for the homogeneous  SSEP (see \cite{MasiSurvey,Ravi1992}) one can deduce that \textbf{(B)} is satisfied, where the law $\mf L$ is determined by the distribution of the Ornstein-Uhlenbeck process at time $r>0$.

We now debate the issue of which  properties  a sequence of initial measures should have in order to  satisfy \textbf{(A)} and \textbf{(B)}.
Assume, for the moment, that the initial measures  $\{\mu_n\}_{n\in \bb N}$ for the Markov processes $\{\eta_{t}: t\geq{0}\}$   satisfy:
\begin{enumerate}[\bf (i)]
\item Condition \eqref{assumptionone} holds.
\item For each $n\in\mathbb N$, the correlation  at the initial time is of order $O(1/n)$ times a bounded profile $\zeta^n$, that is,
\begin{align*}
\varphi^n_0(x,y)\;=\;\frac{\zeta^n(\frac{x}{n},\frac{y}{n})}{n}\,,\qquad\forall\, x,y\in \bb Z\,,\; \forall n\in \bb N\,,
\end{align*} 
where the sequence of functions $\zeta^n:\bb R \times \bb R\to \bb R_+$ converge uniformly to a bounded continuous function $\zeta:\bb R \times \bb R\to \bb R_+$ as $n\to\infty$. Note that this implies \eqref{ass2}.
\end{enumerate}
Under \textbf{(i)} and \textbf{(ii)}, condition \textbf{(A)} holds. Moreover, under \textbf{(i)} and \textbf{(ii)}, and following the  same steps of Subsection~\ref{subsec3.5}, one can obtain tightness of $\{\mc Y^n_0\}_{n\in \bb N}$. Hence, in order to achieve \textbf{(B)}, it is  only  missing the convergence in distribution of the sequence of  initial density fields  $\{\mc Y^n_0\}_{n\in \bb N}$. Let $f,g\in \mc S_\alpha(\bb R)$. By simple calculations, 
\begin{align*}
\bb E_{\mu_n}\Big[\mc Y_0^n(f)\mc Y_0^n(g)\Big]\;=\; & \frac{1}{n}\sum_{x\in \bb Z}f(\pfrac{x}{n})g(\pfrac{x}{n})\,\bb E_{\mu_n}\Big[\big(\overline{\eta}_0(x)\big)^2\Big]+\frac{1}{n}\sum_{\topo{x\neq y}{x,y\in \bb Z}}f(\pfrac{x}{n})g(\pfrac{y}{n})\varphi^n_0(x,y)\,.
\end{align*}
Above $\bar\eta$ denotes the centered random variable $\eta$ : $\bar\eta_t(x):=\eta_t(x)-\rho_t^n(x)$.
Under \textbf{(i)} and \textbf{(ii)}, it is easy to check that expression above converges to
\begin{align*}
\int_{\mathbb{R}} \chi(\rho_0(u))f(u) g(u)  \,du+\int_{\bb R}\int_{\bb R} \zeta(w,r) f(w) g(r)\,dw\,dr
\end{align*}
as $n\to\infty$. Note that the limit above indicates that $1/n$ is the right order on the decay of correlations  in order to exist  a limiting non zero effect on the distribution of initial density field $\mc Y_0$. However, convergence of means and decay of correlations do not suffice to assure that $\mc Y_0^n(f)$ actually converges in distribution: some special central limit theorem is required here. This CLT is not an easy subject due to the slow decay of correlations and due to the fact that for each $n$, the random variables $\eta_0(x)$ may have different distributions. 
We therefore leave it as an open question:
\begin{open}\label{conjecture} Given assumption {\rm \textbf{(A)}} of Theorem~\ref{thm26}, which 
 additional hypotheses are necessary for {\rm \textbf{(B)}} to hold?
\end{open}

Without going into details, we affirm that a natural strategy to prove  current/tagged particle fluctuations  relies in a decay of correlations  of order\break $O(1/n)$, see \cite{JL_2006}. However, the correlations of the non equilibrium SSEP with a slow bond here considered are of order $O(\log n/n)$, see Theorem~\ref{Prop25}. Moreover, the current/ tagged particle fluctuations for the equilibrium scenario with a slow bond are already understood, see \cite{fgn2}. This  leads us to:
\begin{open}\label{conjecture2} How to prove current/tagged particle fluctuations for the non equilibrium SSEP with a slow bond? May (or must) a  different scaling  be considered?
\end{open}
Finally, naturally inspired by \cite{fgn2}, we state:
\begin{open}\label{conjecture3} Consider $\beta >0$ with $\beta\neq 1$.
How to prove non-equilibrium fluctuations for the one-dimensional SSEP with a slow bond of rate $\alpha n^{-\beta}$?
\end{open}
We believe that this last open problem shall be solved by the methods of this paper, and we leave it for a future work. For the first two problems, we have no clear strategy to solve it.

\section{Estimates on local times}\label{s5}
In this section we derive estimates on the local times of a random walk with inhomogeneous rates, which will be later used in Section~\ref{s4} in the proofs of Theorems~\ref{prop24} and~\ref{Prop25}. In the sequel, given any Markov chain $Z$ and a set $A$, we denote by $L_t(A)$ the local time of $Z$ in $A$ until time $t$:
\begin{equation}\label{eq:local_time}
L_t(A)\;\overset{\text{def}}{=}\;\int_0^t\textbf{1}_{\{Z_s\in A\}}ds\,.
\end{equation}
\subsection{Estimates in dimension two}\label{sec:est_dim_two}
We denote by $\{({\bf X}_t,{\bf Y}_t);\; t\geq 0\}$ the random walk on the set $V=\{(x,y)\in \bb Z\times \bb Z:\,y\geq x+1\}$ with generator ${\bf B}_n$ acting on local functions $f:V\to\bb R$ via
\begin{equation} \label{an}
\big({\bf {B}}_n f\big)(u)\;\overset{\text{def}}{=}\;\sum_{v\in V}c_n(u,v)\big[f(v)-f(u)\big]\,,\,\,\, \forall \, u\in V\,.
\end{equation}
Here, the rates are defined as pictured in Figure~\ref{fig2}. More precisely, for $u=(u_1,u_2)$ and $v=(v_1,v_2)$ such that  the $L^1$-norm\footnote{We write $\Vert \cdot \Vert_1$ for the $L^1$-norm on $\bb Z^2$, that is, $\Vert (u_1,u_2)\Vert_1 = |u_1|+|u_2|$.
} satisfies   $\Vert u-v\Vert_1 =1$, we define
\begin{equation*}
c_n(u,v)\;\overset{\text{def}}{=}\;\left\{\begin{array}{cl}
\frac{\alpha}{n}, &  \mbox{if}\,\,\,\,(u,v)\in \mc U,\\
1, &  \mbox{if}\,\,\,\,u\notin U \text{ or } v\notin U.
\end{array}
\right.
\end{equation*}
and $c_n(u,v)=0$ if the $L^1$-distance of $u$ and $v$ is not equal to one.
Here,
$U$ is the subset of $V$ given by
\begin{align*}
U \;\overset{\text{def}}{=}\; & \big\{(x,y)\in V :\, x\in\{0,1\} \textrm{ and } y\geq 2\big\}\cup  \big\{(x,y)\in V :\, x\leq -1 \textrm{ and } y\in\{0,1\}\big\}\,,
\end{align*}
and $\mc U$ is the subset of $U^{\otimes 2}$ defined via
\begin{equation}
\begin{aligned}
\mc U\;\overset{\text{def}}{=}\;&\big\{(u,v)\in U^{\otimes 2}:\, \Vert u-v \Vert=1,\, |u_1-v_1|=1 \text{ and }u_2, v_2\geq 2 \big\}\\
\cup\,&\big\{(u,v)\in U^{\otimes 2}:\, \Vert u-v \Vert=1,\, |u_2-v_2|=1 \text{ and } u_1, v_1\leq -1\big\}\,.
\end{aligned}
\end{equation}
We furthermore denote by $D$ the ``upper diagonal'' defined by 
$ D\;\overset{\text{def}}{=}\;\big\{(x,y)\in \bb Z^2 \,:\, y=x+1\big\}\,,$
see Figure~\ref{fig2} below.
\begin{figure}[H]
\centering
\begin{tikzpicture}[scale=0.7]
\filldraw[fill=black!15!, draw=lightgray] (-4,-3.5)--(-4,4)--(3.5,4)--cycle;
\draw[step=0.5cm,gray,very thin] (-4,-4) grid (4,4);
\draw[->,  thick] (-4.5,0)--(4.5,0) node[below]{$x$};
\draw[->, thick] (0,-4.5)--(0,4.5) node[left]{$y$};
\draw[very thick,rounded corners,dotted] (-4,-3.5)--(3.5,4);
\foreach \x in {2,...,8}
 \draw[ultra thick, black] (0,0.5*\x cm) -- (0.5,0.5*\x cm);
\foreach \x in {2,...,8}
\shade[ball color=gray](0,0.5*\x cm) circle (3pt);
 \foreach \x in {2,...,8}
\shade[ball color=gray](0.5,0.5*\x cm) circle (3pt);
 \foreach \x in {1,...,8}
   \draw[ultra thick, color=black] (-0.5*\x cm, 0) -- (-0.5*\x cm, 0.5);
\foreach \x in {1,...,8}
\shade[ball color=gray](-0.5*\x cm,0) circle (3pt);
 \foreach \x in {1,...,8}
 \shade[ball color=gray](-0.5*\x cm, 0.5) circle (3pt);
\end{tikzpicture}
\caption{Sets $V$, $D$ and $U$ and $\mc U$. Sites of $V$ are the ones laying on the light gray region.  Sites in $D$ lay on the dotted line and  sites of $U$ are marked as gray balls.
Elements of $\mc U$ are edges marked with a thick black segment having jump rate equal to $\alpha/n$ (slow bonds). Any other edges have rate $1$.}
\label{fig2}
\end{figure}
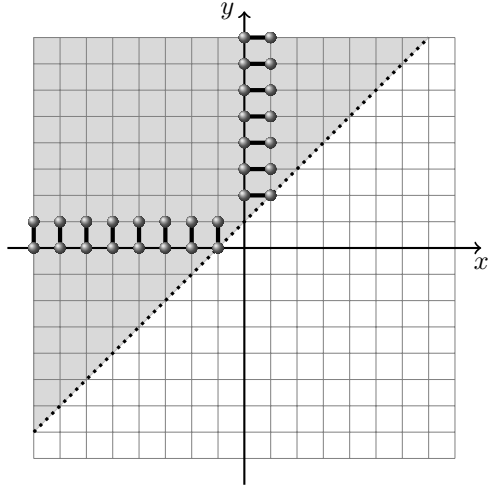
\noindent By $\bf E_{(x,y)}$, and $\bf P_{(x,y)}$ we  denote the corresponding probability and expectation when starting from $(x,y)\in V$.
The goal of this section is to prove the following result.
\begin{proposition}
	\label{prop:2dlocaltimes}
	There exists a constant $c>0$ such that for all $(x,y)\in V$, all $n\in\bb N$, and all $t\geq 0$,
	\begin{equation}
	\begin{aligned}
	{\bf E}_{(x,y)}\Big[L_{tn^2}  \Big(D\backslash\{(0,1)\}\Big)\Big]  &\;\leq\; cn\sqrt{t}, \quad\text{and}\\
	 {\bf E}_{(x,y)}\Big[L_{tn^2}\Big( \{(0,1)\}\Big)\Big]  &\;\leq\; c\log (tn^2).	
	\end{aligned}
	\end{equation}
\end{proposition}
To prove Proposition~\ref{prop:2dlocaltimes}, we estimate first in Lemma~\ref{lem:localtimetriangle} the local time of a simple random walk confined to the boundary of the set $$W\;\overset{\text{def}}{=}\;\{(x,y)\in \bb Z^2:\, 0\leq x\leq y\}\,,$$
which is $V$ intersected with the first quadrant shifted by the vector $(1,2)$. In plain words we identified the vertex $(1,2)$ in $V$ with the origin, which is a change only of notational nature. Its proof consists on  a comparison argument, which is the content of Proposition~\ref{prop:monotonicity}. Afterwards, in Lemma~\ref{lem:crossings}, we show that the expected number of jumps over the set of slow bonds (i.e., those with rates $\alpha/n$) is finite. Finally, with all that at hand, we are able to finish the proof.

We denote by $({\bf X}^{\mytriangle}, {\bf Y}^{\mytriangle})$ the continuous time simple random walk on $W$ that jumps from a site $z_1\in W$ to any fixed neighbouring site $z_2\in W$ at rate $1$, i.e., the simple random walk reflected at the boundary of $W$ (which takes a triangular shape, see Figure~\ref{fig3b}). In particular the total jump rate out of $z_1\in W$ is equal to the number of nearest neighbours of $z_1$ that lay inside $W$. Expectation with respect to $({\bf X}^{\mytriangle}, {\bf Y}^{\mytriangle})$ conditioned to start at $(x,y)\in W$ is denoted by ${\bf E}_{(x,y)}^{\mytriangle}$. 
\begin{figure}[H]
	\centering
	\begin{tikzpicture}
	\begin{scope}[scale=0.7,xshift=0cm]
	\filldraw[fill=lightgray, draw=lightgray] (0,0)--(0,6)--(6,6)--cycle;
	\draw[step=1cm,gray,very thin] (0,0) grid (6,6);
	\filldraw[fill=white, draw=white] (0,0)--(6.2,0)--(6.2,6)--(6,6)--cycle;
	\draw[->, thick] (-0.5,0)--(6.5,0) node[below]{$x$};
	\draw[->, thick] (0,-0.5)--(0,6.5) node[left]{$y$};
	\draw[ thick] (0,0)--(6,6);
	
	\begin{scope}[xshift=1cm,yshift=1cm]
	\draw[very thick,->] (3,3)--(3,4);
	\draw[very thick,->] (3,3)--(2,3);
	\draw (2.4,3) node[below]{\small $1$};
	\draw (3,3.6) node[right]{\small $1$};
	\end{scope}
	
	\begin{scope}[xshift=-3cm,yshift=0cm]
	\draw[very thick,->] (3,3)--(3,4);
	\draw[very thick,->] (3,3)--(3,2);
	\draw[very thick,->] (3,3)--(4,3);
	\draw (3,2.3) node[right]{\small $1$};
	\draw (3,3.7) node[right]{\small $1$};
	\draw (3.7,3) node[below]{\small $1$};
	\end{scope}

	\end{scope}
	\end{tikzpicture}
	\caption[Illustration]{For  $({\bf X}^{\mytriangle},{\bf Y}^{\mytriangle})$
		any jump rate is equal to $1$.}
	\label{fig3b}
\end{figure}
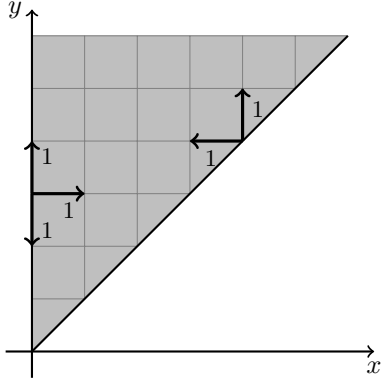
\begin{lemma}\label{lem:localtimetriangle}
	There exists a constant $c>0$ such that for all $t\geq 0$ and all $(x,y)\in W$,
	\begin{equation}
	{\bf E}_{(x,y)}^{\mytriangle}\big[L_t(\partial W)\big]\;\leq\; c\sqrt{t}\,.
	\end{equation}
\end{lemma}
To prove the above lemma we will need two additional results. To introduce the first one, we remind the reader that a continuous time Markov chain on a countable set $\mc E$ can be constructed from a transition probability $p$ on $\mc E$ and a bounded function $\lambda:\mc E\to (0,\infty)$ as follows:
\begin{itemize}
	\item[(1)] sample a discrete time Markov chain $(\xi_n)_{n\geq 0}$ with transition probability~$p$;
	\item[(2)] sample a sequence of independent random variables $(\tau_n)_{n\geq 0}$ such that $\tau_n$ is exponentially distributed with rate $\lambda(\xi_n)$ and define the successive sequence of jump times via $T_0=0$ and $T_n=\tau_n +T_{n-1}$ for $n\geq 1$;
	\item[(3)] finally, define the continuous time Markov chain $Z$ via $$Z_t\;=\;\xi_n\textbf{1}_{\{T_n\leq t< T_{n+1}\}}\,.$$
\end{itemize} 
To continue, we fix a transition probability $p$ on $\mc E$ and for any $a,b$ such that $0 < a\leq b <\infty$, we  denote by $Z^{[a,b]}$ the continuous time Markov chain with transition probability $p$ and such that its field of rates $(\lambda_{[a,b]}(x))_{x\in G}$ is such that $\lambda_{[a,b]}(x)\in[a,b]$ for all $x\in G$. We denote the expectation with respect to $Z^{[a,b]}$ started in $z\in \mc E$ by ${\bf E}_z^{[a,b]}$.
\begin{proposition}\label{prop:monotonicity}
Fix $0 < a<b <c < d <\infty$, and define $\Lambda\overset{\text{def}}{=}\sup_{x\in \mc E}\frac{\lambda_{[c,d]}(x)}{\lambda_{[a,b]}(x)}$.  For any  $A\subseteq \mc E$ and any $z\in \mc E$,
\begin{equation}\label{claim}
{\bf E}_z^{[c,d]}\big[L_t(A)\big]\;\leq\; {\bf E}_z^{[a,b]}\big[L_{\Lambda t}(A)\big]\,.
\end{equation}
\end{proposition}
The second result is about projections (also called lumping) of continuous time Markov chains.
\begin{proposition}\label{prop:lumping}
Let $\mc E$ be a countable set, and consider a bounded function $\zeta:\mc E\times \mc E\to[0,\infty)$. Let $(X_t)_{t\geq 0}$ be the continuous time Markov chain with state space $\mc E$  and jump rates $\{\zeta(x,y)\}_{x,y\in \Omega}$. 
Fix an equivalence relation $\sim $ on $\mc E$ with equivalence classes $\mc E^{\sharp}=\{[x]:\, x\in \mc E\}$ and assume that $\xi$ satisfies
\begin{align}\label{eq46}
\sum_{y'\sim y} \zeta (x,y')\;=\; \sum_{y'\sim y} \zeta (x', y')
\end{align}
whenever $x\sim x'$. Then, $([X_t])_{t\geq 0}$ is a Markov chain with state space $\mc E^{\sharp}$  and jump rates $\zeta ([x], [y])= \sum_{y'\sim y} \zeta (x,y')$.
\end{proposition}
We first prove Proposition~\ref{prop:monotonicity}, afterwards Proposition~\ref{prop:lumping} and finally we prove Lemma~\ref{lem:localtimetriangle}.
\begin{proof}[Proof of Proposition~\ref{prop:monotonicity}]
To prove \eqref{claim} we use a coupling argument.
We do so by first sampling the discrete time Markov chain $(\xi_n)_{n\geq 0}$ as alluded above, and we intend to construct $Z^{[a,b]}$ and $Z^{[c,d]}$ both from the same realization of $(\xi_n)_{n\geq 0}$.
To that end, we consider an independent field of Poisson clocks $(N_x^{[c,d]})_{x\in \mc E}$ such that for any $x\in \mc E$ the rate of $N_x^{[c,d]}$ equals $\lambda_{[c,d]}(x)$. We further define $$N_x^{[a,b]}(t)\;\overset{\text{def}}{=}\; N_x^{[c,d]}\Big(\frac{t}{\lambda_{[c,d]}(x)}\lambda_{[a,b]}(x)\Big)$$ and it readily follows that for each $x\in \mc E$ the process $N_x^{[a,b]}$ is a Poisson process with rate $\lambda_{[a,b]}(x)$.
Hence, it follows from the construction outlined before the statement of Proposition~\ref{prop:monotonicity} that the construction above yields indeed a coupling of $Z^{[a,b]}$ and $Z^{[c,d]}$.

This coupling has the following two properties, which immediately proves \eqref{claim}. Denote by $Z_{[0,t]}^{[c,d]}$ the sequence of visited points by the process $Z_{[0,t]}$ until time $t$, with an analogous definition for $Z^{[a,b]}_{[0,\Lambda t]}$.
\begin{itemize}
	\item[(1)] There exists some $u\in [0,\Lambda t]$ such that $Z_{[0,t]}^{[c,d]}=Z^{[a,b]}_{[0,u]}$. That is, the sequence $Z_{[0,t]}^{[c,d]}$ is an initial piece of  $Z^{[a,b]}_{[0,\Lambda t]}$.
	\item[(2)] Given $x \in Z_{[0,t]}^{[c,d]}$, then at its $k$-th visit to $x$ the holding time at that point of $Z^{[a,b]}$ is larger than the one of $Z^{[c,d]}$.
\end{itemize}
	\end{proof}
	\begin{proof}[Proof of Proposition~\ref{prop:lumping}]
Let $P$ be the transition matrix of the skeleton chain of $(X_t)_{t\geq 0}$ (i.e., of the underlying discrete time Markov chain). Assumption \eqref{eq46} implies that
\begin{align*}
P(x, [y])\;=\; P(x', [y])
\end{align*}
 whenever $x\sim x'$. It then follows from \cite[Lemma 2.5, pp. 25]{yuval} that the skeleton chain of $([X_t])_{t\geq 0}$ is a discrete time Markov chain with transition matrix given by  $P^{\sharp}([x], [y]):=P(x, [y])$. Thus, it remains to show that the holding times of $([X_t])_{t\geq 0}$ are exponentially distributed with rates $\Big\{\sum_{[y]}\zeta([x],[y])\Big\}_{[x]\in \mc E^\sharp}$. Yet, this is, as well, a consequence of \eqref{eq46}. Hence, we can conclude the proof.
\end{proof}
\begin{proof}[Proof of Lemma~\ref{lem:localtimetriangle}]
	The proof comes in two steps.
		\begin{figure}[H]
	\centering
	\begin{tikzpicture}
	\begin{scope}[scale=0.7]
	\filldraw[fill=lightgray, draw=lightgray] (0,0)--(0,6)--(6,6)--(6,0)--cycle;
	\draw[step=1cm,gray,very thin] (0,0) grid (6,6);
	\draw[->, thick] (-0.5,0)--(6.5,0) node[below]{$x$};
	\draw[->, thick] (0,-0.5)--(0,6.5) node[left]{$y$};
	\begin{scope}[xshift=1cm,yshift=1cm]
	\draw[very thick,->] (3,3)--(3,4);
	\draw[very thick,->] (3,3)--(3,2);
	\draw[very thick,->] (3,3)--(4,3);
	\draw[very thick,->] (3,3)--(2,3);
	\draw (3,2.25) node[left]{$\pfrac{1}{2}$};
	\draw (3,3.75) node[right]{$\pfrac{1}{2}$};
	\draw (2.3,3) node[above]{$\pfrac{1}{2}$};
	\draw (3.7,3) node[below]{$\pfrac{1}{2}$};
	\end{scope}
	
	\begin{scope}[xshift=-3cm,yshift=-1cm]
	\draw[very thick,->] (3,3)--(3,4);
	\draw[very thick,->] (3,3)--(3,2);
	\draw[very thick,->] (3,3)--(4,3);
	\draw (3,2.2) node[right]{$\pfrac{1}{2}$};
	\draw (3,3.8) node[right]{$\pfrac{1}{2}$};
	\draw (3.8,3) node[below]{$\pfrac{1}{2}$};
	\end{scope}
	\draw (3,-1.5) node[align=left]{Rates of $\big({\bf X}^{\square},{\bf Y}^{\square}\big)$ are \\ everywhere equal to $1/2$.};
	\end{scope}

	\begin{scope}[scale=0.7,xshift=8cm]
	\filldraw[fill=lightgray, draw=lightgray] (0,0)--(0,6)--(6,6)--cycle;
	\draw[step=1cm,gray,very thin] (0,0) grid (6,6);
	\filldraw[fill=white, draw=white] (0,0)--(6.2,0)--(6.2,6)--(6,6)--cycle;
	\draw[->, thick] (-0.5,0)--(6.5,0) node[below]{$x$};
	\draw[->, thick] (0,-0.5)--(0,6.5) node[left]{$y$};
	\draw[ thick] (0,0)--(6,6);
	
	\begin{scope}[xshift=1cm,yshift=1cm]
	\draw[very thick,->] (3,3)--(3,4);
	\draw[very thick,->] (3,3)--(2,3);
	\draw (2.4,3) node[below]{\small $1$};
	\draw (3,3.6) node[right]{\small $1$};
	\end{scope}
	
	\begin{scope}[xshift=-3cm,yshift=0cm]
	\draw[very thick,->] (3,3)--(3,4);
	\draw[very thick,->] (3,3)--(3,2);
	\draw[very thick,->] (3,3)--(4,3);
	\draw (3,2.2) node[right]{$\pfrac{1}{2}$};
	\draw (3,3.8) node[right]{$\pfrac{1}{2}$};
	\draw (3.8,3) node[below]{$\pfrac{1}{2}$};
	\end{scope}
	
	\draw (3,-2) node[align=left]{$\big([{\bf X}^{\square}],[{\bf Y}^{\square}]\big)$ only lives  on the upper \\ triangle. The  rates out of   diagonal \\ are doubled, but  the rates \textit{towards} \\ the  diagonal remain $1/2$.};
	\end{scope}
	\end{tikzpicture}
	\caption{Relation between $\big({\bf X}^{\square},{\bf Y}^{\square}\big)$ and $\big([{\bf X}^{\square}],[{\bf Y}^{\square}]\big)$.}
	\label{fig3}
\end{figure}
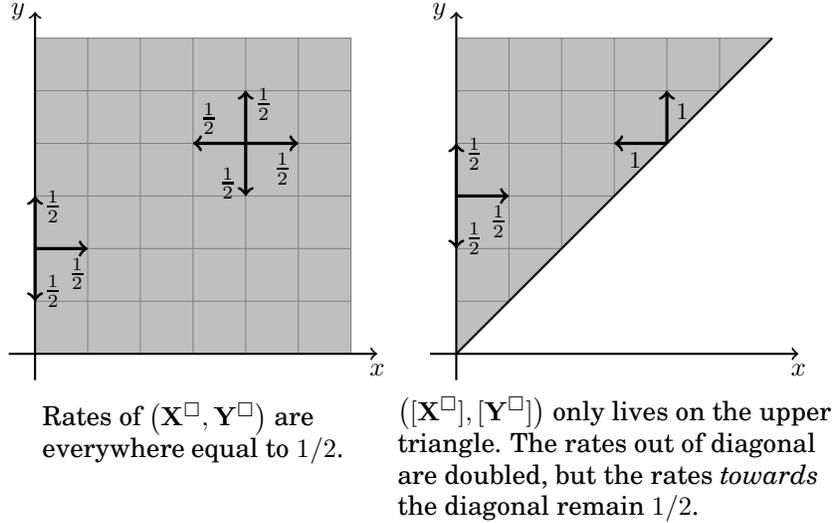

\noindent	\textbf{1st Step.} In this step we show that it is sufficient to estimate the local time of a simple random walk on $$\bb Z_{\geq 0}^2\;\overset{\text{def}}{=}\;\{(x,y)\in\bb Z^2:\, x,y\geq 0\}\,,$$ and we refer the reader to Figure~\ref{fig3b} and~\ref{fig3}  for an illustration of the various random walks that will appear in this part of the proof.
	To that end, let $({\bf X}^{\Box}, {\bf Y}^{\Box})$ be a simple random walk defined on $\bb Z_{\geq 0}^2$ that jumps from $z_1\in\bb Z_{\geq 0}^2$ to a fixed neighbouring site $z_2\in\bb Z_{\geq 0}^2$ at rate $\frac12$. Write
	\begin{equation}
	\Wd\;\overset{\text{def}}{=}\;\{(x,y)\in W:\, x=y\}\,.
	\end{equation}
		Our aim is to show that for any $(x,y)\in W$, 
	\begin{equation}
	\label{eq:localtimeXbar}
	{\bf E}_{(x,y)}^{\mytriangle}\big[L_t(\partial W)\big]\;\leq\; {\bf E}_{(x,y)}^{\Box}\big[L_{2 t}(\partial \bb Z_{\geq 0}^2)\big] + {\bf E}_{(x,y)}^{\Box}\big[L_{2 t}(\Wd)\big]\,,
	\end{equation}where the expectations on the right hand side of  the display above denote the  expectation with respect to  $({\bf X}^{\Box}, {\bf Y}^{\Box})$ started at $(x,y)$.
	To see that~\eqref{eq:localtimeXbar} is true we consider the function $T: \bb Z_{\geq 0}^2\to \bb Z_{\geq 0}^2$ that maps each $z\in \bb Z_{\geq 0}^2$ to its reflection with respect to the diagonal $\Wd$. Note in particular that $T$ is its own inverse, so that we can define an equivalence relation on $\bb Z_{\geq 0}^2$ via
	\begin{equation}
	z_1\sim z_2 \quad \Longleftrightarrow\quad \exists\, n\in \{1,2\}\mbox{ such that } T^n(z_1)=z_2\,. 
	\end{equation}
	Writing $\bf{P}^{\Box}$ for the transition matrix corresponding to the underlying discrete time random walk of $({\bf X}^{\Box}, {\bf Y}^{\Box})$, and by $\{\zeta(x,y)\}_{x,y\in \bb Z_{\geq 0}^2}$ its field of rates, it is easy to see that for any $z_1\sim z_2$ and any $z_3\in\bb Z_{\geq 0}^2$
	\begin{equation}
	\sum_{z_3'\sim z_3}\zeta(z_1,z_3')\;=\;\sum_{z_3'\sim z_3}\zeta(z_2, z_3')\,.
	\end{equation}
	Hence, by Proposition~\ref{prop:lumping} and the fact that the set $\bb Z_{\geq 0}^2/\!\!\sim$ equals $W$, the process $([{\bf X}^{\Box}], [ {\bf Y}^{\Box}])$ can be identified with a simple random walk on $W$ such that its jump rate out of each fixed edge equals $\frac12$ except for those attached to $\Wd$, where the jump rate is $1$. Note in particular that this makes the set of edges directed. Indeed, the jump rate from $x\in \Wd$ to any neighbour $y$ is $1$, whereas the jump  rate  from $y$ to $x$ is $\frac12$. 
	
	Denoting by ${\bf E}_{([x],[y])}^{[\Box]}$ the expectation  with respect to $([{\bf X}^{\Box}], [ {\bf Y}^{\Box}])$ when\break started at $([x],[y])$, we see that as a consequence of Proposition~\ref{prop:monotonicity},
	\begin{equation}
{\bf E}_{(x,y)}^{\mytriangle}\big[L_t(\partial W)\big]\;\leq\; {\bf E}_{([x],[y])}^{[\Box]}\big[L_{2t}([\partial W])\big]\,.
	\end{equation}
	Thus,~\eqref{eq:localtimeXbar} readily follows from last  inequality.\bigskip
	
		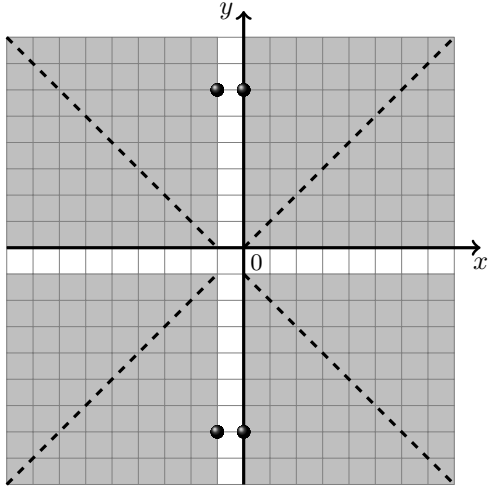
\begin{figure}[!htb]
	\centering
	\begin{tikzpicture}[scale=0.7]
	\filldraw[fill=lightgray, draw=lightgray] (0,0)--(4,0)--(4,4)--(0,4)--cycle;
	
	\begin{scope}[xshift=-4.5cm]
	\filldraw[fill=lightgray, draw=lightgray] (0,0)--(4,0)--(4,4)--(0,4)--cycle;
	\end{scope}
	
	\begin{scope}[xshift=-4.5cm,yshift=-4.5cm]
	\filldraw[fill=lightgray, draw=lightgray] (0,0)--(4,0)--(4,4)--(0,4)--cycle;
	\end{scope}
	
	\begin{scope}[xshift=0cm,yshift=-4.5cm]
	\filldraw[fill=lightgray, draw=lightgray] (0,0)--(4,0)--(4,4)--(0,4)--cycle;
	\end{scope}
	
	\draw[step=0.5cm,gray,very thin] (-4.5,-4.5) grid (4,4);
	
	\draw[->, very thick] (-4.5,0)--(4.5,0) node[below]{$x$};
	\draw[->, very thick] (0,-4.5)--(0,4.5) node[left]{$y$};

	\begin{scope}[xshift=0cm]
	\draw[very thick, dashed] (-4.5,-4.5)--(-0.5,-0.5);
	\draw[very thick, dashed] (0,0)--(4,4);
	\draw[very thick, dashed] (-4.5,4)--(-0.5,0);
	\draw[very thick, dashed] (0,-0.5)--(4,-4.5);
	\end{scope}
	
	\draw (-0.05,0.05) node[anchor= north west]{\small $0$};

	\shade[ball color=black](0,3) circle (0.125);
	\shade[ball color=black](0,-3.5) circle (0.125);
	\shade[ball color=black](-0.5,3) circle (0.125);
	\shade[ball color=black](-0.5,-3.5) circle (0.125);
	
	\end{tikzpicture}
	\caption{Ilustration of equivalence relation in the \textbf{2nd Step} of the proof of Lemma~\ref{lem:localtimetriangle}. $\Wd$ gets identified with  the points on dashed lines. The four points marked with black balls compose a single equivalence class.   Non-zero jump rates between any two equivalence classes are everywhere equal to~$1/2$.}
	\label{fig4}
\end{figure}

\noindent	\textbf{2nd Step.} We now show that it is sufficient to estimate certain local times of a simple random walk $(X, Y)$ on $\Z^2$ jumping at total rate $2$ (i.e., the jump rate over any fixed edge is $\frac12$), which will then yield the claim.
	To that end we define an equivalence relation by imposing that  $(x,y) \sim (x,-y-1)$ and $(x,y)\sim (-x-1,y)$, for any $x,y\in \bb Z$.
	We then note that in this way $\Wd$ gets identified with 
	\begin{align*}
		\big[\Wd\big]\;\overset{\text{def}}{=}\; & \big\{(x,y)\in \bb Z^2:\, x=y\big\}\\
		&\cup\big\{(x,y)\in \bb Z^2:\, x=-y-1,\, y\geq 0\, \mbox{ or } y=-x-1,\, x\geq 0\big\}\,,
	\end{align*}
	see  Figure~\ref{fig4}.
   	Note that by Proposition~\ref{prop:lumping} the random walk $({\bf X}^{\Box},{\bf Y}^{\Box})$ can be identified with $([{ X}],[{ Y}])$. 
	This shows that it is sufficient to bound
	\begin{equation}
	{\bf E}_{(x,y)}^{(X,Y)}\big[L_{2t}(A)\big]\,,
	\end{equation}
	where
	$
	A\;=\;A_1\cup A_2\cup A_3$
	with
	\begin{equation}
	\begin{aligned}
	A_1&\;\overset{\text{def}}{=}\;\big\{(x,y)\in \bb Z^2:\, x=y\big\}\,,\\
	A_2&\;\overset{\text{def}}{=}\;\big\{(x,y)\in \bb Z^2:\, x=-y-1,\, y\geq 0\, \mbox{ or } y=-x-1,\, x\geq 0\big\}\,, \quad\mbox{and}\\
	A_3&\;\overset{\text{def}}{=}\;\big\{(x,y)\in \bb Z^2:\, x\in\{0,-1\}, \mbox{ or }y\in\{0,-1\}\big\}\,.
	\end{aligned}
	\end{equation}
Since $X-Y$ has the same law as a one-dimensional symmetric simple random walk, we conclude that $L_{2t}(A_1)$ equals in law to the local time at zero of a one-dimensional symmetric simple random walk, for which the statement of this lemma is well known, and for completeness, we provide a short proof of it in Proposition~\ref{prop:1dwalk}. A similar argument may be used for $A_2$, and $A_3$. Therefore, we can  finish the proof.
\end{proof}

We now come back to the original problem, i.e., estimating local times of the random walk $(\mb X,\mb Y)$ defined on the set $V$.
An important ingredient in the analysis will be an estimate on the number of jumps of $(\mb X,\mb Y)$ over the set of slow edges, i.e., those that are depicted with thick black segments in Figure~\ref{fig2}. We denote the set of these edges by $\mc S$, and we define a sequence of stopping times via
\begin{equation}\label{eq:tau_stop_time}
\begin{aligned}
\tau_1&\;=\;\inf\big\{t\geq 0:\, (\mb X_t, \mb Y_t) \mbox{ crossed an edge in }\mc S\big\},\quad \mbox{and for }i\geq 2,\quad\\
\tau_i&\;=\;\inf\big\{t\geq \tau_{i-1}:\, (\mb X_t, \mb Y_t) \mbox{ crossed an edge in }\mc S\big\}.
\end{aligned}
\end{equation}
Finally, we  define the number of crossings until the time $tn^2$ via
\begin{equation}
C_{tn^2}\;=\;\sup\big\{i\geq 0:\, \tau_i\leq tn^2\big\}\,.
\end{equation}
\begin{lemma}
\label{lem:crossings}
There exists a constant $c>0$ such that uniformly over all starting points $(x,y)\in V$, all $t\geq 0$, and all $n\in \bb N$,
\begin{equation}
{\bf E}_{(x,y)}[C_{tn^2}]\;\leq\; c\,.
\end{equation}
\end{lemma}
\begin{proof}
We first show that for all $i\geq 1$,
\begin{equation}\label{eq:tau1}
\inf_{(x,y)} {\bf P}_{(x,y)}[\tau_i-\tau_{i-1} \geq tn^2] \;>\;0\,.
\end{equation}
To that end, assume without loss of generality that $(x,y)$ is in the first quadrant. In this case $\tau_1$ can be interpreted as a first success of the simple random walk $({\bf X}^{\mytriangle}, {\bf Y}^{\mytriangle})$, which with a slight abuse of notation is now considered on the set $W$ given by the intersection of $V$ with the first quadrant, in the following way: whenever $({\bf X}^{\mytriangle}, {\bf Y}^{\mytriangle})$ is on a vertex $z$ that is attached to a slow bond it realizes the following experiment: besides its three (one if the vertex is $(1,2)$) independent Poisson clocks $N_z^1, N_z^2,$ and $N_z^3$ ringing at rate $1$ that are needed for its graphical construction, it considers an additional independent Poisson clock $N_z(\alpha)$ ringing at rate $\alpha/n$. We then say that the experiment is successful if $N_z(\alpha)$ rings before any of the other three clocks. It then follows from the construction that the time of the first success equals in law the time of the first jump of $(\mb X,\mb Y)$ over a slow bond. Indeed, one may couple $(\mb X, \mb Y)$ and $({\bf X}^{\mytriangle}, {\bf Y}^{\mytriangle})$ such that they move together until the first time of success.
Thus, using the fact that each experiment is independent of the evolution of $({\bf X}^{\mytriangle}, {\bf Y}^{\mytriangle})$, and that the set of vertices that are attached to $\mc S$ is a subset of $\partial W$, we see that for any constant $c\in (0,1)$,
\begin{equation}\label{eq:tau1lowerbound}
\begin{aligned}
{\bf P}_{(x,y)}[\tau_1\geq tn^2]
&\;\geq\; {\bf P}_{(x,y)}^{\mytriangle}\big[L_{tn^2}(\partial W)\leq c\sqrt{t}n, \text{all experiments are unsuccessful}\,\big]\\
&\;\geq\;  {\bf P}_{(x,y)}^{\mytriangle}\big[L_{tn^2}(\partial W)\leq c\sqrt{t}n]\cdot{\bf P}\big[\mathrm{exp(\alpha/n)}\geq c\sqrt{t}n\big]\,,
\end{aligned}
\end{equation}
where $\mathrm{exp(\alpha/n)}$ denotes an exponentially distributed random variable with rate $\alpha/n$. It now follows from Lemma~\ref{lem:localtimetriangle} and Markov's inequality that there exists  $c\in (0,1)$ such that the right hand-side of~\eqref{eq:tau1lowerbound} is strictly larger than zero, uniformly in $(x,y)$. With similar arguments we may derive the same statement for all $i\geq 2$.
We next introduce the random variable
\begin{equation}
N\;=\;\inf\big\{i\geq 1:\, \tau_i-\tau_{i-1}\geq tn^2\big\}\,.
\end{equation}
Then, using the strong Markov property at time $\tau_{i-1}$, bounding the probability of the event $\{\tau_i-\tau_{i-1}\geq tn^2\}$ by $1$, and then once again using the strong Markov property at time $\tau_{i-2}$, we can estimate for any $i\geq 1$,
\begin{equation}\label{eq:N}
\begin{aligned}
{\bf P}_{(x,y)}[N=i]&\;\leq\; {\bf P}_{(x,y)}\Big[\bigcap_{j=1}^{i-1}\{\tau_{j}-\tau_{j-1} < tn^2\}\Big]\\
&\;=\; {\bf E}_{(x,y)}\Big[\prod_{j=1}^{i-2}\textbf{1}_{\{\tau_{j}-\tau_{j-1} < tn^2\}}{\bf E}_{(\mb X_{\tau_{i-2}}, \mb Y_{\tau_{i-2}})}[\textbf{1}_{\{\tau_1 < tn^2\}}]\Big]\,.
\end{aligned}
\end{equation}
Using~\eqref{eq:tau1}, we see that there exists  $c\in[0,1)$ that is independent of the starting point $(x,y)$, such that the latter term above is bounded from above by
\begin{equation}
c\,{\bf E}_{(x,y)}\Big[\prod_{j=1}^{i-2}\textbf{1}_{\{\tau_{j}-\tau_{j-1} < tn^2\}}\Big]\,.
\end{equation}
Iterating the above procedure we can get that
\begin{equation}
\sup_{(x,y)\in V}{\bf P}_{(x,y)}[N=i]\;\leq\; c^{i-1}\,,
\end{equation}
which in turn implies the uniform boundedness in $(x,y)\in V$ of the expectation of $N$. Since $C_{tn^2}\leq N$, this implies the claim.
\end{proof}
We present now  the proof of Proposition~\ref{prop:2dlocaltimes}, and we focus first on the local time of the set $D\setminus\{(0,1)\}$. For definiteness we assume that $(\mb X,\mb Y)$ starts in $(x,y)\in W$, where we recall that, abusing of notation, $W$ denotes $V$ intersected with the first quadrant. All other cases follow by a straightforward adaptation of this proof. Note that the event $\{(\mb X_s,\mb Y_s)\in D\setminus\{(0,1)\}\}$ is only possible, if $s\in \cup_{i=0}^{\infty}[\tau_{2i}, \tau_{2i+1})$, where $\tau_0=0$. Hence, we can write
\begin{equation}\label{eq:localtimeslow}
\begin{aligned}
&{\bf E}_{(x,y)}\Big[L_{tn^2}  \Big(D\backslash\{(0,1)\}\Big)\Big]\;=\;{\bf E}_{(x,y)}\Big[\int_0^{tn^2} \mb 1_{\{(\mb X_s,\mb Y_s)\in D\setminus\{(0,1)\}\}}\, ds\Big]\\
&=\; \sum_{i=0}^{\infty}{\bf E}_{(x,y)}\Big[\int_{\tau_{2i}\wedge tn^2}^{\tau_{2i+1}\wedge tn^2} \mb 1_{\{(\mb X_s,\mb Y_s)\in D\setminus\{(0,1)\}\}}\, ds\Big]\,.
\end{aligned}
\end{equation}
Fix $i\in\bb N$. Applying the strong Markov property at time $\tau_{2i}$ we can rewrite each summand in the display above as
\begin{equation}\label{eq:strongMarkov}
{\bf E}_{(x,y)}\Big[\mb 1_{\{\tau_{2i}<tn^2\}}{\bf E}_{(\mb X_{\tau_{2i}},\mb Y_{\tau_{2i})}}\Big[\int_0^{ \overline{\tau}_1\wedge tn^2-\tau_{2i}}\mb 1_{\{(\overline{\mb X}_s,\overline{\mb Y}_s)\in D\setminus\{(0,1)\}\}}\, ds\Big]\Big]\,,
\end{equation}
where $(\overline{\mb X},\overline{\mb Y})$ denotes an independent copy of $(\mb X,\mb Y)$ and $\overline{\tau}_1$ is the corresponding stopping time, defined in the same way as $\tau_1$ in \eqref{eq:tau_stop_time}. We now recall that as a consequence of the proof of Lemma~\ref{lem:crossings} until the time $\overline{\tau}_1$ the walk $(\mb X,\mb Y)$ can be coupled with $({\bf X}^{\mytriangle}, {\bf Y}^{\mytriangle})$.
 Hence, we see that~\eqref{eq:strongMarkov} is at most
\begin{equation}
{\bf E}_{(x,y)}\big[\mb 1_{\{\tau_{2i}<tn^2\}}\big]
\sup_{(x,y)\in W} {\bf E}_{(x,y)}\Big[ \int_0^{tn^2}\mb 1_{\{({\bf X}^{\mytriangle}_s, {\bf Y}^{\mytriangle}_s)\in D\setminus\{(0,1)\}\}}\, ds\Big]\,.
\end{equation}
Making use of Lemma~\ref{lem:localtimetriangle} we see that there exists a constant $c\in (0,\infty)$ such that for all starting points, all $t\geq 0$ and all $n\in \bb N$, the term on the left hand-side of~\eqref{eq:localtimeslow} is bounded from above by
\begin{equation}
c\sqrt{t}n \,{\bf E}_{(x,y)}\big[C_{tn^2}\big]\,.
\end{equation}
Hence, an application of Lemma~\ref{lem:crossings} is enough to conclude the claim.
To estimate the local time of  the vertex $(0,1)$ we can proceed almost exactly as above, and we see that there exists a constant $c\in(0,+\infty)$ such that
\begin{equation}\label{eq:estin01}
\begin{split}
{\bf E}_{(x,y)}\Big[L_{tn^2}  \big(\{(0,1)\}\big)\Big] &\;=\;\int_0^{tn^2}{\bf P}_{(x,y)}\Big[(\mb X_{s},\mb Y_{s})= (0,1)\Big]\, ds\\
&\;\leq\; c \sum_{z\in \mc A}\int_0^{2tn^2} {\bf P}_{(x,y)}\Big[(X_{s},Y_{s})= z\Big]\, ds\,,
\end{split}\end{equation}
where we recall that $(X,Y)$ denotes the simple random walk on $\bb Z^2$ jumping at total rate $2$, and
$\mc A=\{(0,1), (1,1), (0,0), (1,0)\}$. The proof now follows from the local central limit theorem, see for instance~\cite[Theorem 2.5.6]{LawlerLimic2010} (this result is stated for one-dimensional continuous time random walks, however using the fact  that a $d$-dimensional continuous time random walk consists of $d$ independent one-dimensional random walks, it may be easily adapted to our setting), or alternatively from Proposition~\ref{prop:2dwalk}.

\subsection{Estimates in dimension one}\label{sec:est_dim_one}
We denote by $\{\mb X_t;\; t\geq 0\}$ the random walk on $\bb Z$ with a slow bond, that is, the random walk with  infinitesimal generator $\mca_n$ given in \eqref{op_B} and we use ${\bf E}_{x}, {\bf P}_{x}$ to denote the corresponding expectation and probability, starting from $x\in \bb Z$.
\begin{lemma}\label{lem:heatkernel}
For all $x,y\in \bb Z$, and for all $t\geq 0$ we have the equality
\begin{equation}
\label{eq:heatkernel}
{\bf P}_x\big(\mb X_t=y\big)+{\bf P}_x\big(\mb X_t=-y+1\big) \;=\;{\bf P}_x\big(X_t=y\big)+{\bf P}_x\big(X_t=-y+1\big)\,,
\end{equation}
where $(X_t)_{t\geq 0}$ denotes a one-dimensional symmetric simple random walk jumping at total rate $2$.
\end{lemma}
\begin{proof}
The proof comes in two steps.\smallskip 

\noindent
\textbf{1st Step.} In this step we rewrite the left hand-side in~\eqref{eq:heatkernel} in terms of the transition probabilities of a symmetric simple  random walk that is reflected at $1$.
To that end, we define the following equivalence relation:
\begin{equation}
x\sim y \quad\Longleftrightarrow\quad y=-x+1\text{ or } y=x\,.
\end{equation}
Note that in particular in this way $0$ gets identified with $1$, so that jumps between these two vertices ``do not count''. One may then readily check that condition~\eqref{eq46} is satisfied, so that $([\mb X_t])_{t\geq 0}$ defines a continuous time Markov chain. It is then plain to see that  for all $x\in\bb Z$ and all $t\geq 0$ the following relation holds:
\begin{equation}
[\mb X_t]=[x]\quad\Longleftrightarrow\quad \mb X_t\in\{x,-x+1\}\,.
\end{equation} 
Thus,
\begin{equation}\label{eq:equivalencewalk}
{\bf P}_x(\mb X_t\in \{y,-y+1\})\;=\; {\bf P}_{[x]}([\mb X_t]=[y])\,.
\end{equation}
Choosing only representants in the set $\bb Z_{\geq 1}\overset{\text{def}}{=}\{x\in\bb Z:\, x\geq 1\}$ we see that $([\mb X_t])_{t\geq 0}$ may be identified with a simple random walk $(X_t^{R})_{t\geq 0}$ on $\bb Z_{\geq 1}$ that jumps from any vertex $x\in \bb Z_{\geq 1}$ to a fixed neighboring vertex in $\bb Z_{\geq 1}$ at rate $1$. 
Thus, for any $x,y\geq 1$,~\eqref{eq:equivalencewalk} becomes
\begin{equation}\label{eq:reflectedwalk}
{\bf P}_x(\mb X_t\in \{y,-y+1\})\;=\; {\bf P}_{x}(X^{R}_t=y)\,.
\end{equation}

\noindent \textbf{2nd Step.} In this step we show that the right hand-side of~\eqref{eq:reflectedwalk} may be rewritten in terms of a symmetric simple  random walk on $\bb Z$ jumping at total rate $2$.
To that end we use the same equivalence relation as above and we note that $(X_t)_{t\geq 0}$ can be, in the same way, identified with $(X_t^R)_{t\geq 0}$ as $(\mb X_t)_{t\geq 0}$ can be identified with $([\mb X_t])_{t\geq 0}$. This finishes the proof.
\end{proof}

\section{Estimates on the discrete derivative and correlations}\label{s4}

In the next two subsections, we present the proofs of Theorems \ref{prop24} and \ref{Prop25}, respectively.

\subsection{Estimate on the discrete derivative} 
 
This section is devoted to the proof of Theorem~\ref{prop24}.
\begin{proof}
Recall  that $\rho_t^n$ is a solution of \eqref{disc_heat}.  Since the statement is clear for $x=0$, we only need to deal with the case $x\neq 0$. 
Let $\rho_t$ be the solution of the equation \eqref{robin}, and define $\gamma^n:[0,T]\times\bb Z^d\to\bb R$ via 
\begin{equation}
\gamma_t^n(x)\;\overset{\text{def}}{=}\;
\begin{cases}
\rho_t^n(x)-\rho_t\big(\pfrac{x}{n}\big), &\text{if }x\neq 0,\\
\rho_t^n(0)-\rho_t\big(\pfrac{-1}{n^2}\big), &\text{otherwise.}
\end{cases}
\end{equation}
The reason for the previous definition is that it distinguishes two cases,  since at $x=0$ the time derivative of $\rho$ is not related  to its spatial derivatives in a way that is helpful for our purposes. However, with the above choice of $\gamma^n$ we see that for all $x\in\bb Z$,
\begin{equation} \label{gamma}
\p_t\gamma_t^n(x)\;=\;n^2\mca_n\gamma_t^n(x)+F_t^n(x)\,,
\end{equation}
where
\begin{equation}
F_t^n(x)\;\overset{\text{def}}{=}\;
\begin{cases}
\big(n^2\mca_n-\p_u^2\big)\rho_t\big(\pfrac{x}{n}\big) &\text{if }x\neq 0,\smallskip \\
n^2\mca_n \rho_t(0)-\p_u^2\rho_t\big(\pfrac{-1}{n^2}\big) & \text{otherwise.}
\end{cases}
\end{equation} 
Observe that, by the definition of $\mc A_n$ in  \eqref{op_B}, for $x\in\bb Z\backslash \{0,1\}$, $F_t^n$ accounts for the difference between the discrete and the continuous Laplacian.
To continue, we add and subtract $\rho_t\big(\pfrac{x}{n}\big)$ and $\rho_t\big(\pfrac{x+1}{n}\big)$ to $\big|\rho_t^n(x+1)-\rho_t^n(x)\big|$ and use the triangle inequality which yields
\begin{equation} \label{rhon}
\big|\rho_t^n(x+1)-\rho_t^n(x)\big|\;\leq\; |\gamma_t^n(x+1)|+|\gamma_t^n(x)|+\big|\rho_t\big(\pfrac{x+1}{n}\big)-\rho_t\big(\pfrac{x}{n}\big)\big|\,.
\end{equation}
We first treat the rightmost term above. Since $x\mapsto\rho_t(x)$ is differentiable in any neighborhood outside of zero, and $\rho_t$ has one sided spatial derivatives at zero, we see that 
\begin{equation*}
\big|\rho_t\big(\pfrac{x+1}{n}\big)-\rho_t\big(\pfrac{x}{n}\big)\big| \;=\; O\big(\pfrac{1}{n}\big)\,.
 \end{equation*}
Recall that $\{\mb X_t;\; t\geq 0\}$ denotes the random walk on $\bb Z$ generated by $\mca_n$.  Applying Duhamel's principle we see that we can write the solution of \eqref{gamma} as 
\begin{equation*}
\gamma_t^n(x)\;=\;{\bf E}_x\Big[\gamma_0^n(\mb X_{tn^2}) +\int_0^t F_{t-s}^n(\mb X_{sn^2})\,ds\Big]\,.
\end{equation*}
Therefore, 
\begin{equation*}
\sup_{t\leq T}\;\sup_{x\in \bb Z}|\gamma_t^n(x)|\;\leq\; \sup_{x\in \bb Z}|\gamma_0^n(x)|\;+\;\sup_{t\leq T}\;\sup_{x\in \bb Z}\Big|
{\bf E}_{x}\Big[\int_0^t F_{t-s}^n(\mb X_{sn^2})\,ds\Big]\Big|\,.
\end{equation*}
Since $|\gamma_0^n(x)|=|\rho_0^n(x)-\rho_0(x)|$, by Assumption~\eqref{assumptionone} we only need to control the second term on the right hand-side of the previous expression. By Fubini's Theorem,  we see that
\begin{equation}\label{exp}
{\bf E}_{x}\Big[\int_0^t F_{t-s}^n(\mb X_{sn^2})\,ds\Big]\;=\; \int_0^t \sum_{z\in\bb Z} {\bf P}_x \big[\mb X_{sn^2}=z\big]F_{t-s}^n(z)\,ds\,.
\end{equation}
Since the discrete Laplacian approximates the continuous Laplacian, we conclude that $|F_t^n(x)|\leq C/n^2$ for any $x\in \bb Z\backslash\{0,1\}$ and for any $t\geq0$. Therefore, we can bound the absolute value of  \eqref{exp} by
\begin{equation} \label{ref1}
t\frac{C}{n^2}\;+\; \int_0^t \sum_{z\in \{0,1\}} {\bf P}_x \big[\mb X_{sn^2}=z\big]\big|F_{t-s}^n(z)\big|\,ds\,.
\end{equation}
Moreover, we also have that 
\begin{equation*}
\begin{split}
F_{t-s}^n(1)&\;=\;n^2\Big(\rho_t\big(\pfrac{2}{n}\big)-\rho_t\big(\pfrac{1}{n}\big) + \frac{\alpha}{n}\big(\rho_t\big(\pfrac{0}{n}\big)-\rho_t\big(\pfrac{1}{n}\big)\big)\Big)-\p_u^2\rho_t\big(\pfrac{1}{n}\big)\\
&\;=\;n\Big(n\big(\rho_t\big(\pfrac{2}{n}\big)-\rho_t\big(\pfrac{1}{n}\big)\big) + \alpha\big(\rho_t\big(\pfrac{0}{n}\big)-\rho_t\big(\pfrac{1}{n}\big)\big)\Big)-\p_u^2\rho_t\big(\pfrac{1}{n}\big)\,.
\end{split}
\end{equation*}
Summing and subtracting $\alpha\rho(0^+)$, using the Robin boundary conditions and Taylor expansion, the last equation becomes bounded from above by
\begin{equation*}
\begin{split}
\Big|n\Big(\pfrac{1}{n}\p_u^2\rho_t\big(\pfrac{1}{n}\big)+O(1/n^2)\Big)+\pfrac{1}{2}\p_u^2\rho_t\big(\pfrac{1}{n}\big)-\alpha\p_u\rho_t(0^+)+O(1/n)-\p_u^2 \rho_t\big(\pfrac{1}{n}\big)\Big|,
\end{split}
\end{equation*}
from where we get that $|F_t^n(1)|\leq C$ for any $t\geq0$. For $z=0$ we obtain, in a similar way, a bound of the same order. Therefore, \eqref{ref1} is bounded from above by
\begin{equation*}
t\frac{C}{n^2}\;+\; C\int_0^t \big({\bf P}_x \big[\mb X_{sn^2}=0\big]+{\bf P}_x \big[\mb X_{sn^2}=1\big]\big)\,ds\,.
\end{equation*}
Thus, applying Lemma~\ref{prop42} below the result follows.
\end{proof}
\begin{lemma}\label{prop42}
Let $\mb X$ be as in Subsection \ref{sec:est_dim_one}. There exists  a constant $C>0$ such that the following estimate holds for all $t\geq 0$:
\begin{equation*}
\int_0^t{\bf P}_x \Big[\mb X_{sn^2}\in\{0,1\}\Big]\,ds\;\leq \; \frac{C\sqrt{t}}{n}\,.
\end{equation*}
\end{lemma} 
\begin{proof}
Denote the symmetric simple random walk on $\bb Z$ jumping at rate $2$ by $\{X_t;\, t\geq 0\}$. It is then well known that for all $t\geq 0$ the map $x\in\bb Z\mapsto {\bf P}_x[X_t=0]$ is maximized at $x=0$. Hence, Lemma~\ref{prop42} is a consequence of Lemma~\ref{lem:heatkernel} together with Proposition~\ref{prop:1dwalk}. 
\end{proof}

\subsection{Estimate on the  correlation function}
In this section we prove Theorem~\ref{Prop25}. To that end, we show that the correlation function $\varphi^n$ introduced in Definition~\ref{defcorr} can be estimated from above by the local times of the random walk $\{({\bf X}_t,{\bf Y}_t);\; t\geq 0\}$, introduced in Subsection \ref{sec:est_dim_two}. This is the content of Proposition~\ref{prop:decorrprop}. Proposition~\ref{prop:2dlocaltimes} then immediately yields the result. 
Given a set $A\subseteq V$, similarly as in Section~\ref{s5} we denote by $L_t(A)$ the local time of $\{(\mb X_t, \mb Y_t);\, t\geq 0\}$ until time $t$ in $A$, see \eqref{eq:local_time}.
\begin{proposition}\label{prop:decorrprop}
There exists $C>0$ such that 
\begin{equation}\label{eq:covariancelocaltime}
\begin{aligned}
&\sup_{t\leq T}|\varphi_t^n(x,y)|\\
&\leq \frac{C}{n}+C\Big(\frac{1}{{n^2}}({\bf E}_{(x,y)}[L_{n^2 T}(D\setminus\{(0,1)\})]+\frac{1}{n}{\bf E}_{(x,y)}[L_{n^2 T}(\{(0,1)\})]\Big) .
\end{aligned}
\end{equation}
\end{proposition}
\begin{proof}
First, observe that from Kolmogorov's forward equation, we have that 
\begin{equation*}
\p_t\varphi_t^n(x,y)\;=\;\bb E_{\mu_n}\big[n^2 \mc L_n(\eta_t(x)\eta_t(y))\big]-\p_t(\rho_t^n(x)\rho_t^n(y))\,.
\end{equation*}
Applying \eqref{ln} and \eqref{rho_t} and performing some long, but simple, calculations, one can deduce that $\varphi_t^n$ solves the following equation:
\begin{equation*}
\p_t\varphi_t^n(x,y)\;=\;n^2 {\bf B}_n\varphi_t^n(x,y)+g_t^n(x,y)\,,
\end{equation*}
where ${\bf B}_n$ was defined in \eqref{an} and  
\begin{equation} \label{gnt}
g_t^n(x,y)\;=\;- (\nabla_n^+\rho_t^n(x))^2\Big(\textbf{1}_{\{ D\backslash(0,1)\}}+\frac{\alpha}{n}\textbf{1}_{\{(0,1)\}}\Big)\,.
\end{equation}
Here, $\nabla_n^+$ denotes the rescaled discrete right derivative which, for any function $f:\bb Z\to\bb R$, is defined via $\nabla_n^+f(x)\overset{\text{def}}{=}n(f(x+1)-f(x))$.
By Duhamel's Principle,
\begin{equation*}
\varphi_t^n(x,y)\;=\;{\bf E}_{(x,y)}\Big[\varphi_0^n(\mb X_{tn^2},\mb Y_{tn^2}) +\int_0^tg_{t-s}^n(\mb X_{sn^2},\mb Y_{sn^2})\,ds\Big],
\end{equation*}
where $\{({\bf X}_t,{\bf Y}_t);\; t\geq 0\}$ is the random walk with generator ${\bf B}_n$.  
In order to prove the proposition we just have to estimate the right hand-side of the last equation. We see that
\begin{equation}\label{eq:varphiest}
\sup_{t\leq T}|\varphi_t^n(x,y)|\leq |\varphi_0^n(x,y)|+\sup_{t\leq T}\Big|
{\bf E}_{(x,y)}\Big[\int_0^tg_{t-s}^n(\mb X_{sn^2},\mb Y_{sn^2})\,ds\Big]\Big|.
\end{equation}
By Assumption \eqref{ass2}, the first term on the right hand-side of the last expression is bounded from above by $C/n$. Thus, to finish the proof we only need to estimate the rightmost term in the display above.

Applying the definition of $g^n$, and rewriting the expectation above in terms of transition probabilities, we see that for any $s\leq t$,
\begin{align*}
{\bf E}_{(x,y)}[g_{t-s}^n(\mb X_{sn^2},\mb Y_{sn^2})] &=
\sum_{z\neq 0}\big[- (\nabla_n^+\rho_{t-s}^n(z))^2\big]{\bf P}_{(x,y)}[(\mb X_{sn^2},\mb Y_{sn^2})=(z,z+1)]\\
&+\frac{\alpha}{n}\big[- (\nabla_n^+\rho_{t-s}^n(0))^2\big] {\bf P}_{(x,y)}[(\mb X_{sn^2},\mb Y_{sn^2})=(0,1)]\,. 
\end{align*}
Consequently, for all $(x,y)\in V$, the rightmost term in~\eqref{eq:varphiest} is bounded from above by
\begin{equation} \label{Snn}
\begin{split}
S_n\!\int_0^t\!\!\Big({ \bf P}_{(x,y)}[(\mb X_{sn^2},\mb Y_{sn^2})\in D\!\setminus\!\{(0,1)\}]
+ S_{n,0}\frac{\alpha}{n}\, {\bf P}_{(x,y)}[(\mb X_{sn^2},\mb Y_{sn^2})=(0,1)]\Big)ds,
\end{split}
\end{equation}
where 
\begin{equation} \label{Sn}
S_n\;=\;\sup_{t\geq 0}\,\sup_{z\in\bb Z\backslash\{0\}} \,(\nabla_n^+\rho_t^n(z))^2\,\,\,\,\,\,\, \textrm{and}\,\,\,\,\,\,\, S_{n,0}\;=\;\sup_{t\geq 0}\,\,(\nabla_n^+\rho_t^n(0))^2\,.
\end{equation}
Recalling Theorem~\ref{prop24}, we easily deduce that
$S_n\leq C$ and $S_{n,0}\leq Cn^2$. 
Substituting~\eqref{Sn} into~\eqref{Snn}, together with a change of variables, the result follows.
\end{proof} 
The proof of Theorem~\ref{Prop25} is now an immediate consequence of Proposition~\ref{prop:2dlocaltimes}.

\subsection{Comments on the lower bound}\label{rm27}
In the usual symmetric simple exclusion process the correlation function is of order  $O(\tfrac1n)$. Since intuitively one could expect that the presence of the slow bond increases the correlation between sites which are located both on the positive half-line or both the negative half-line, the above result does not come as a total surprise.

 However, for two sites $x$ and $y$ such that $x\leq 0 < 1 \leq y$, then at first sight it seems to be a reasonable guess that the correlations decrease, and they should  be at most of order $O(\tfrac1n)$. Yet, our proof yields the same bound as above when one restricts only to such kind of pairs of vertices $(x,y)$. A natural question therefore is if a matching lower bound in~\eqref{eq:covarianceest} holds. Since our assumptions on the initial measure do not exclude the choice of a product Bernoulli measure with constant intensity, in which case at any time $t\geq 0$ the covariance between two distinct points is zero, such a lower bound certainly cannot hold in general.
 
   Nevertheless, we argue that there are indeed choices of the various parameters in our model for which $|\varphi_t^n(x,y)|$ is bounded from below by a constant times $\log n/n$ uniformly in $t\in [0,T]$. We will not provide all the details, yet the gaps can be filled by an adaptation of the techniques used in Section~\ref{s5}. We choose
$\mu_n \sim \otimes_{x\in \bb Z} \mathrm{Ber}(\rho_x)$, where
\begin{equation}
\rho_x\;=\;
\begin{cases}
\pfrac12, &\text{if }x\leq 0,\\
\pfrac14, &\text{otherwise.}
\end{cases}
\end{equation}
Analyzing carefully the proof of Theorem~\ref{Prop25}, we see that in order to establish the desired lower bound it is enough to show that there exists a constant $c>0$ such that for all $t\in [0,T]$
\begin{equation}
|\rho_t^n(0)-\rho_t^n(1)|\;\geq\; c\,,
\end{equation}
and that the rightmost local time term in~\eqref{eq:covariancelocaltime} is bounded from below by a constant times $\log n$. We only provide a sketch of the argument for the former statement, the latter as mentioned above can be deduced by an application of the techniques developed in Section~\ref{s5}.
We note that it is possible to show that
\begin{equation}
\rho_t^n(0)\;=\; \sum_{x\in \bb Z}\bb P_0\big[\mb X_t=x\big]\rho_0^n(x) \quad\text{and} \quad
\rho_t^n(1)\;=\; \sum_{x\in \bb Z}\bb P_1\big[\mb X_t=x\big]\rho_0^n(x),
\end{equation}
where $\mb X$ denotes a random walk with generator $n^2\mc A_n$, and for $z\in \bb Z$ we denoted by $\bb P_z$ the distribution of $\mb  X$ when started in $z$.
Using that by symmetry $\bb P_1[\mb X_t\geq 1]=\bb P_0[\mb X_t\leq 0]$ and $\bb P_1[\mb X_t\leq 0]=\bb P_0[\mb X_t\geq 1]$, as well as our choice of $\mu_n$, we see that
\begin{equation}\label{eq:strictlowerboundrho}
\rho_t^n(0)-\rho_t^n(1)\;=\; \pfrac14\Big(\bb P_0\big[\mb X_t\leq 0\big]-\bb P_0\big[\mb X_t\geq 1\big]\Big)\,.
\end{equation}
It is now possible to argue that a random walk that starts at zero, and that is reflected at zero has a local time of order $n$ up to times of order $n^2$ at the origin. Using a coupling argument one may then show that one can choose $\alpha$ small enough so that the probability that $\mb X$, when started at $0$, crosses the bond $(0,1)$ before time $Tn^2$ becomes arbitrarily small. This readily yields that~\eqref{eq:strictlowerboundrho} is indeed strictly bounded away from zero uniformly in $t\in[0,T]$, and consequently we obtain a lower bound that matches the order of the upper bound in~\eqref{eq:covarianceest}.

\begin{remark}\rm\label{rm28}
As argued above, at first sight it seems counterintuitive that $\varphi_t(x,y)$ is of order $\log n/n$ if $x\leq 0< 1\leq y$. Yet, an intuitive explanation for that phenomenon could be as follows: given an exclusion particle starting at $x\leq0$, then up to time say $\frac{t}{2} n^2$ there is a strictly positive probability that it will cross the bond $\{0,1\}$, and afterwards it will have interaction with  a particle started at $y\geq 1$ of same order as if it had started at a site $x\geq 1$.
\end{remark}

\section{Proof of density fluctuations}\label{s3}
In this section we prove Theorem~\ref{thm26}. We follow the usual procedure to establish such a result, i.e., first we establish tightness of the sequence of density fields $\{\mc Y_t^n\!:\!t\in[0,T]\}_{n\in\bb N}$  and afterwards we characterize the limit.
Before proceeding,  we introduce in the next subsection some martingales associated with the density fluctuation field defined in \eqref{density field}.

\subsection{Associated martingales}
Fix a test function $f\in\mc S_\alpha(\bb R)$. By Dynkin's formula, 
\begin{equation}\label{martingal}
\mc M^n_t(f)\;:=\;\mc Y^n_t(f)- \mc Y^n_0(f)-\int_{0}^{t}(n^{2}\mc L_{n}+\partial_s)\,\mc Y^n_s(f)\,ds
\end{equation}
is a martingale with respect to the natural filtration $\mathcal{F}_{t}=\sigma(\eta_{s},s\leq{t})$. Our aim is to write this martingale in a more suitable form.  Recall \eqref{xi}.  Performing simple calculations,\allowdisplaybreaks
\begin{align*}
& n^2\mc L_{n}\mc Y^n_s(f)=\\
&=n^2 \sum_{x\in{\bb Z}}\xi_{x, x+1}^n\Big[ \pfrac{1}{\sqrt{n}}\sum_{y\in\bb Z} f(\pfrac{y}{n})(\eta^{x,x+1}_s (y)-\rho_s^n(y) )-
 \pfrac{1}{\sqrt{n}}\sum_{y\in\bb Z} f(\pfrac{y}{n})(\eta_s (y)-\rho_s^n(y) )\Big]\\
&= \pfrac{1}{\sqrt{n}}\sum_{x\in{\bb Z}}n^2\xi_{x, x+1}^n\Big\{\eta_s (x)\Big[ f(\pfrac{x+1}{n})
- f\big(\pfrac{x}{n}\big) \Big]+\eta_s (x+1)\Big[f\big(\pfrac{x}{n}\big)
- f(\pfrac{x+1}{n}) \Big]\Big\}\\
&= \pfrac{1}{\sqrt{n}}\sum_{x\in{\bb Z}}n^2\Big\{\xi_{x, x+1}^n\Big[f(\pfrac{x+1}{n})
- f\big(\pfrac{x}{n}\big) \Big]+\xi_{x-1,x}^n\Big[ f(\pfrac{x-1}{n})
- f\big(\pfrac{x}{n}\big) \Big]\Big\}\eta_s (x)\\
&= \pfrac{1}{\sqrt{n}}\sum_{x\in{\bb Z}}n^2\mca_n f\big(\pfrac{x}{n}\big)\eta_s(x)\,,
\end{align*}
where  the operator $\mca_n$  has been defined in \eqref{op_B}.
Recalling \eqref{disc_heat} we get that
\begin{equation}\label{partial Y}
 \begin{split}
  \partial_s\,\mc Y^n_s(f)\;=\; - \pfrac{1}{\sqrt{n}}\sum_{x\in{\bb Z}}f\big(\pfrac{x}{n}\big)\partial_s\rho_s^n(x) \!=\!-\pfrac{1}{\sqrt{n}}\sum_{x\in{\bb Z}}n^2\mca_n f\big(\pfrac{x}{n}\big)\rho_s^n(x)\,.
 \end{split}
\end{equation}
Combining the previous equalities, we see that
\begin{equation}\label{eq:mart_decom}
\mc M^n_t(f)\;=\;\mc Y^n_t(f)- \mc Y^n_0(f)-\int_{0}^{t}\pfrac{1}{\sqrt{n}}\sum_{x\in{\bb Z}}n^2\mca_n f\big(\pfrac{x}{n}\big)\overline{\eta}_s(x)\,ds.
\end{equation}
 Adding and subtracting the term $\int_0^t \mc Y^n_s(\Delta_\alpha f)ds$, we can rewrite the martingale $\mc M^n_t(f)$ as
\begin{equation}\label{martdecomp3.4}
\mc M^n_t(f)\;=\;\mc Y^n_t(f)- \mc Y^n_0(f)-\int_{0}^{t}\mc Y^n_s(\Delta_\alpha f)ds-R_t^{n}(f)\,,
\end{equation}
where
\begin{equation*}
R_t^{n}(f)\;:=\;\int_{0}^{t}\frac{1}{\sqrt{n}}\sum_{x\in\bb Z}\Big\{n^2\mca_n f\big(\pfrac{x}{n}\big)-(\Delta_\alpha f)\big(\pfrac{x}{n}\big)\Big\}\overline{\eta}_{s}(x)\,ds\,.
\end{equation*}
The next lemma allows us to control the error term  $R_t^{n}(f)$ defined in the previous display, which is obtained by replacing the discrete operator $\mc A_n$ defined in \eqref{op_B} by the continuous Laplacian $\Delta_\alpha$ defined in \eqref{laplacian}.
\begin{lemma}\label{lemma31aa}
For any $f\in\mc S_\alpha (\bb R)$, almost surely there exists a constant $c>0$ such that  for all $t\in[0,T]$ and all $n\in \bb N$ the estimate $|R_t^{n}(f)|\leq \frac{ct}{\sqrt{n}}$ holds.
\end{lemma}
\begin{proof}
We begin by splitting $R_t^{n}(f)$ as the sum
\begin{align}
R_t^n(f)\;=\;&\int_{0}^{t}\frac{1}{\sqrt{n}}\sum_{x\neq 0,1}\Big\{ n^2\mca_{n}f\big(\pfrac{x}{n}\big)-(\Delta_\alpha f)\big(\pfrac{x}{n}\big)\Big\}\overline{\eta}_{s}(x)\,ds\label{34}\\
&+\int_{0}^{t}\frac{1}{\sqrt{n}}\Big\{ n^2\mca_{n}f\big(\pfrac{0}{n}\big)-(\Delta_\alpha f)\big(\pfrac{0}{n}\big)\Big\}\overline{\eta}_{s}(0)\,ds\label{35}\\
&+\int_{0}^{t}\frac{1}{\sqrt{n}}\Big\{ n^2\mca_{n}f\big(\pfrac{1}{n}\big)-(\Delta_\alpha f)\big(\pfrac{1}{n}\big)\Big\}\overline{\eta}_{s}(1)\,ds\,.\label{36}
\end{align}
We begin by dealing with~\eqref{34}.
Recall that $f\in \mc S_\alpha(\bb R)$ and note that $|\overline{\eta}_{s}(x)|\leq 2$. Thus, taking advantage of the fact that for $x\notin\{0,1\}$, the term $n^2\mca_n f\big(\pfrac{x}{n}\big)$ is  the discrete Laplacian, and  applying a Taylor expansion up to second order with the Lagrangian form of the remainder, we see that~\eqref{34} is bounded by
 \begin{equation*}
\begin{split}
&\frac{t}{\sqrt{n}}\sum_{x\neq0,1}\Big|n^2\Big\{\Big[\pfrac{1}{n}f'\big(\pfrac{x}{n}\big)+\pfrac{1}{2n^2}f''\big(\pfrac{x}{n}\big)+\pfrac{f'''\big(\vartheta^+(\frac{x}{n})\big)}{3!n^3}\Big]\Big.\\
&\hspace{2.5cm}\Big.-\Big[\pfrac{1}{n}f'\big(\pfrac{x}{n}\big)-\pfrac{1}{2n^2}f''\big(\pfrac{x}{n}\big)+\pfrac{f'''\big(\vartheta^-(\frac{x}{n})\big)}{3!n^3}\Big]\Big\}-\big(\Delta_\alpha f\big)\big(\pfrac{x}{n}\big)\Big|\\
&=\frac{t}{\sqrt{n}}\sum_{x\neq0,1}\Big|\pfrac{f'''\big(\vartheta^+(\frac{x}{n})\big)}{3!n^3}-\pfrac{f'''\big(\vartheta^-(\frac{x}{n})\big)}{3!n^3}\Big\}\Big|\,,
\end{split}
\end{equation*}
where $\vartheta^+(\frac{x}{n})\in[\frac{x}{n},\frac{x+1}{n}]$ and  $\vartheta^-(\frac{x}{n})\in[\frac{x-1}{n},\frac{x}{n}]$.
Since $f'''$ is  integrable, we conclude that \eqref{34} is of order $O(tn^{-5/2})$, and vanishes as  $n$ tends to infinity.  Since $\Delta_\alpha f$ is bounded, we can see  that the sum of \eqref{35} and \eqref{36} is equal to
\begin{equation*}
\int_{0}^{t}\frac{1}{\sqrt{n}}\Big\{n^2\mca_n f\big(\pfrac{0}{n}\big)\Big\}\overline{\eta}_{s}(0)\,ds+\int_{0}^{t}\frac{1}{\sqrt{n}}\Big\{n^2\mca_n f\big(\pfrac{1}{n}\big)\Big\}\overline{\eta}_{s}(1)\,ds\,
\end{equation*}
plus a term of order $O(\frac{t}{\sqrt{n}})$. Applying the definition of $\mca_n$, the  expression above is equal to
\begin{equation*} 
\begin{split}
&\int_{0}^{t}\frac{n^2}{\sqrt{n}}\Big\{\frac{\alpha}{n}\Big(f\big(\pfrac{1}{n}\big)-f\big(\pfrac{0}{n}\big)\Big)+\Big(f\big(\pfrac{-1}{n}\big)-f\big(\pfrac{0}{n}\big)\Big)\Big\}\overline{\eta}_{s}(0)\,ds\\
+&\int_{0}^{t}\frac{n^2}{\sqrt{n}}\Big\{\frac{\alpha}{n}\Big(f\big(\pfrac{0}{n}\big)-f\big(\pfrac{1}{n}\big)\Big)+\Big( f\big(\pfrac{2}{n}\big)-f\big(\pfrac{1}{n}\big)\Big)\Big\}\overline{\eta}_{s}(1)\,ds\,,
 \end{split}
\end{equation*}
and  we can see that the absolute value of  expression above is bounded by
\begin{equation}\label{eq3.7}
\begin{split}
&t\sqrt{n}\Big\{\,\Big|\alpha\Big(f\big(\pfrac{1}{n}\big)-f\big(\pfrac{0}{n}\big)\Big)+n\Big(f\big(\pfrac{-1}{n}\big)-f\big(\pfrac{0}{n}\big)\Big)\Big|\\
&\hspace{2cm}+\Big|\alpha\Big(f\big(\pfrac{0}{n}\big)-f\big(\pfrac{1}{n}\big)\Big)+n\Big( f\big(\pfrac{2}{n}\big)-f\big(\pfrac{1}{n}\big)\Big)\Big|\,\Big\}\,.
\end{split}
\end{equation}
 Since  $f\in\mc S_\alpha(\bb R)$, we have the boundary conditions 
$\alpha \big(f(0^+)-f(0^-)\big)=\p_u f(0^+)$ $=\p_u f(0^-)$ and also that $f$ is left continuous  at zero, hence
\begin{align*}
& f\big(\pfrac{1}{n}\big)-f\big(\pfrac{0}{n}\big)\;=\;\Big[f(0^+)-f(0^-)\Big]+O(1/n)\,,\\
&n\Big[f\big(\pfrac{-1}{n}\big)-f\big(\pfrac{0}{n}\big)\Big]\;=\;-\p_u f(0^-)+O(1/n)\,,\\
& f\big(\pfrac{0}{n}\big)-f\big(\pfrac{1}{n}\big)\;=\;-\Big[f(0^+)-f(0^-)\Big]+O(1/n)\,,\\
&n\Big[f\big(\pfrac{2}{n}\big)-f\big(\pfrac{1}{n}\big)\Big]\;=\;\p_u f(0^+)+O(1/n)\,,
\end{align*}
which permits to conclude that \eqref{eq3.7} is of order $O(\tfrac{t}{\sqrt{n}})$, finishing the proof.
\end{proof}
Now we study the convergence of the sequence of martingales $\{\mc M^n_t(f)\!:\!t\in[0,T]\}_{n\in\bb N}$. This is the content of the next lemma. 
\begin{lemma}\label{lemma32}
For any $f\in\mc S_\alpha(\bb R)$, the sequence of martingales $\{\mc M^n_t(f):t\in [0,T]\}_{n\in \bb N}$ converges in  distribution under the topology of $\mc D([0,T], \bb R)$, as $n\to\infty$, to a mean-zero Gaussian process $\{\mathcal M_t(f):t\in [0,T]\}$ of  quadratic variation given by 
\begin{equation}\label{quadvar}
\begin{split}
\<\mc M(f)\>_t\;=\;& \int_0^t \int_{\bb R}2\chi(\rho_s(u))(\nabla_\alpha f(u))^2\, du\,ds\\
 +&\int_0^t \Big[\rho_s(0^-)(1-\rho_s(0^+))+\rho_s(0^+)(1-\rho_s(0^-))\Big]\nabla_\alpha f(0^+)\,ds\,.
\end{split}
\end{equation}
\end{lemma}
\begin{proof}
The proof of this lemma consists  on applying \cite[Theorem VIII.3.12, page 473]{js}. According to that theorem, we  have to check:
\begin{enumerate}[\bf i)]
\item condition (3.14), defined  in \cite[page 474]{js},
\item condition [$\hat\delta_5$-D], defined in \cite[3.4, page 470]{js},
\item condition [$\gamma_5$-D], defined in \cite[3.3, page 470]{js}.
\end{enumerate}

By~\cite[Assertion VIII.3.5, page 470]{js}, both 
conditions  [$\hat\delta_5$-D] and  (3.14) are a consequence of
\begin{equation}\label{eq3.9limit}
\lim_{n\to\infty} \bb E_{\mu_n} \Big[ \sup_{s\le t} \big| \mc M^n_s(f) -
\mc M^n_{s-}(f)\big| \Big]\;=\; 0.
\end{equation}
To show~\eqref{eq3.9limit},  note that only two sites of the configuration $\eta$  change its values when a jump occurs. Therefore, 
\begin{equation*}
\sup_{s\le t} \big| \mc M^n_s(f) -
\mc M^n_{s-}(f)\big| \;=\; \sup_{s\le t} \big| \mc Y^n_s(f) -
\mc Y^n_{s-}(f)\big|\; \leq \;\frac{2\|f\|_\infty}{\sqrt n}\,,
\end{equation*}
leading to \eqref{eq3.9limit}. It remains to check Condition 
[$\gamma_5$-D], i.e., the convergence in probability of the quadratic variation of $\mc M_t(f)$, which is given by 
\begin{equation*}
\langle \mathcal{M}^n(f) \rangle_t \;=\; \int_0^t n^2 \Big[\mathcal{L}_n\mathcal{Y}^n_s(f)^2-2\mathcal{Y}^n_s(f)\mc L_n \mathcal{Y}_s^n(f)\Big]\, ds\,.
\end{equation*}
After some elementary computations, the right hand-side of the display above can be rewritten as
\begin{equation}\label{eq311aaa}
\begin{split}
&\int_0^t \frac{1}{n}\sum_{x\neq0}(\eta_s(x)-\eta_s(x+1))^2\Big[n\Big(f\big(\pfrac{x+1}{n}\big)-f\big(\pfrac{x}{n}\big)\Big)\Big]^2\,ds\\
+&\alpha\int_0^t (\eta_s(0)-\eta_s(1))^2\Big(f\big(\pfrac{1}{n}\big)-f\big(\pfrac{0}{n}\big)\Big)^2\,ds\,.
\end{split}
\end{equation}
 which is an \textit{additive functional} of the exclusion process $\eta_t$.  
 It is almost folklore in the literature that~Theorem~\ref{thm21} together with a suitable \textit{Replacement Lemma} and standard computations yield that 
\eqref{eq311aaa} converges in distribution to the right hand-side of \eqref{quadvar} as $n\to\infty$. 
Since this is not the main issue of the proof, and since such a Replacement Lemma under the slow bond's presence has been studied in previous works (as in \cite[Lemma 5.4]{fgn1} for instance), we do not present the proof of this result with  full details, but only  a sketch instead.

By a \textit{Replacement Lemma} we mean a result allowing to replace the time integral 
of the occupation number $\eta_t(x)$ by an average on a box around $x$. The only difference with respect to the usual Replacement Lemma (see \cite{kl}), is the fact that we should avoid an intersection between { this box} and the slow bond in our setting. Hence, we define
 \begin{equation*}
 \eta^{\ell}(x) \;=\; \begin{cases}\displaystyle
 \frac{1}{\ell} \sum_{y=x}^{x+\ell-1}\eta(y)\,,& \quad \text{ for } x\geq 1\,,\vspace{5pt}\\
\displaystyle\frac{1}{\ell} \sum_{y=x-\ell+1}^{x}\eta(y)\,, &\quad \text{ for } x\leq 0\,,
 \end{cases}
  \end{equation*}
which is  related to the side limits appearing in \eqref{quadvar}. 
Taking into account these definitions, the fact that $\eta_t(x)^2=\eta_t(x)$, and the boundary condition of $f$ at zero, one can show that the limit in distribution of \eqref{eq311aaa} is in fact the right hand-side of \eqref{quadvar}.

Since the right hand-side of \eqref{quadvar} is deterministic, the convergence in distribution implies the convergence  in probability, and this  finishes  the proof of the lemma.
\end{proof}

\subsection{Tightness}\label{subsec3.5}
Let $\mc S$ be a Frech\'et space (see \cite{reedsimon} for a definition of a Frech\'et space) and denote by $\mc S'$ its topological dual. We cite here the following useful criterion:
\begin{proposition}[Mitoma's criterion,  \cite{Mitoma}] \label{mitoma}\quad 
A sequence  of processes $\{x_t;t \in [0,T]\}_{n \in \bb N}$  in $\mc D([0,T],\mc {S}')$ is tight with respect to the
Skorohod topology if, and only if, the sequence $\{x_t(f);t \in [0,T]\}_{n \in \bb N}$ of real-valued processes is tight with
respect to the Skorohod topology of $\mc D([0,T], \bb R)$, for any $f \in \mc {S}$.
\end{proposition}
Since $\mc S_\alpha(\bb R)$ is a Frech\'et space (see \cite{fgn2}), tightness of the density field is reduced to showing tightness of a family of real-valued processes. For that purpose, let $f\in{\mc S}_\alpha(\bb R)$. Since the sum of tight processes is also tight, in order to prove tightness of  $\{\mc Y_t^n(f): t \in [0,T]\}_{n \in \bb N}$ it is enough to prove tightness of the remaining processes  appearing in \eqref{martdecomp3.4}, namely
$\{\mc Y_0^n(f)\}_{n \in
\bb N}$, $\{ \int_{0}^t\mc Y_s^n(\Delta_\alpha f)\, ds: t \in [0,T]\}_{n \in \bb N}$, $\{\mc M_t^n(f
): t \in [0,T]\}_{n \in \bb N}$ and  $\{R_t^n(f
): t \in [0,T]\}_{n \in \bb N}$. We deal with all of them separately.

Observe that
\begin{equation*}
\begin{split}
\bb E_{\nu_{\rho_0}^n(\cdot)}\Big[\Big(\mathcal Y_0^n(f)\Big)^2\Big]&\;=\;\frac{1}{n}\sum_{x\in\bb Z}f^2\Big(\frac xn\Big)\chi(\rho^n_0(x))+\frac 2n\sum_{x<y}f\Big(\frac xn\Big)f\Big(\frac yn\Big)\varphi^n_0(x,y)
\end{split}
\end{equation*}
is bounded. As a consequence of  Assumption~\textbf{(B)} in Theorem~\ref{thm26} the sequence of initial conditions $\mc Y_0^n$ converges, therefore it is also tight.

By Lemma~\ref{lemma31aa}, the sequence of  processes $\{R_t^n(f
): t \in [0,T]\}_{n \in \bb N}$ is negligible, thus it is  tight.

By Lemma \ref{lemma32} the sequence of martingales $\{\mc M_t^n(f
): t \in [0,T]\}_{n \in \bb N}$  converges, hence it is  tight as well.

It remains to prove tightness of the integral terms $\{ \int_{0}^t\mc Y_s^n(\Delta_\alpha f)\, ds: t \in [0,T]\}_{n \in \bb N}$. At this point we invoke \textit{Aldous' criterion}:
\begin{proposition}[Aldous' criterion]
 A sequence $\{x_t^n: t\in [0,T]\}_{n \in \bb N}$ of real-valued processes is tight with respect to the Skorohod topology of $\mc
D([0,T],\bb R)$ if:
\begin{enumerate}[\bf i)]
\item
$\displaystyle\lim_{A\rightarrow{+\infty}}\;\limsup_{n\rightarrow{+\infty}}\;\mathbb{P}\Big(\sup_{0\leq{t}\leq{T}}|x_{t
}^n |>A\Big)\;=\;0\,,$

\item for any $\varepsilon >0\,,$
 $\displaystyle\lim_{\delta \to 0} \;\limsup_{n \to {+\infty}} \;\sup_{\lambda \leq \delta} \;\sup_{\tau \in \mc T_T}\;
\mathbb{P}(|
x_{\tau+\lambda}^n- x_{\tau}^n| >\varepsilon)\; =\;0\,,$
\end{enumerate}
where $\mc T_T$ is the set of stopping times bounded by $T$.
\end{proposition}

We first check the first item of Aldous' criterion. By the Cauchy-Schwarz inequality,
\begin{align*}
&\mathbb{E}_{\mu^n}\Big[\sup_{t\leq {T}}\Big(\int_{0}^t\mc Y_s^n(\Delta_\alpha f)\, ds\Big)^2\Big]\\
&\leq\; T \int_{0}^T \mathbb{E}_{\mu^n}\Big[\Big(\frac{1}{\sqrt{n}}\sum_{x\in \bb Z}\Delta_\alpha f(\tfrac{x}{n})(\eta_{s}(x)-\rho^n_s(x))\Big)^2\Big]\, ds\,.
\end{align*}
Observe that the  right hand-side of the display  above is bounded  by $T^2$ times
\begin{equation}\label{estimate}
\frac{1}{{n}}\sum_{x\in \bb Z}\big(\Delta_\alpha f(\tfrac{x}{n})\Big)^2\sup_{t\leq{T}}\chi(\rho^n_t(x))
+\frac{2}{{n}}\sum_{x<y}\Delta_\alpha f(\tfrac{x}{n})\Delta_\alpha f(\tfrac{y}{n})\sup_{t\leq{T}}\varphi^n_t(x,y)\,,
\end{equation}
where $\chi(\rho^n_t(x))$ was defined above \eqref{Lambda} and $\varphi^n_t(x,y)$ is given in Definition~\ref{defcorr}. Since $f\in \mc S_\alpha(\bb R)$, the first term in~\eqref{estimate} may be easily shown to be bounded in $n$. As for the second term the estimate provided by Theorem~\ref{Prop25} is unfortunately not quite enough. Yet, Proposition~\ref{prop:decorrprop} in combination with Proposition~\ref{prop:2dlocaltimes} show that for some constants $c_1,c_2>0$ that do not depend on $t$, and $(x,y)$ we have that for all $n\in\bb N$,
\begin{equation}\label{eq:estofvarphi}
 \varphi^n_t(x,y)\;\leq\; \frac{c_1}{n}+\frac{c_2}{n}\int_0^{Tn^2} {\bf P}_{(x,y)}\Big[(\mb X_{s},\mb Y_{s})= (0,1)\Big]  \, ds\,,
\end{equation}
where $\{({\bf X}_t,{\bf Y}_t);\; t\geq 0\}$ is defined in  Subsection \ref{sec:est_dim_two}.
Plugging the first term on the right hand-side of the display above  into the second term in~\eqref{estimate} gives the desired estimate. To deal with the second term on the right hand-side of \eqref{eq:estofvarphi} we use the fact that by~\eqref{eq:estin01} we can estimate the integral term from above by
\begin{equation}
c \sum_{z\in \mc A}\int_0^{2Tn^2} \bb P_{(x,y)}\Big[(X_{s},Y_{s})= z\Big]\, ds\,,
\end{equation}
where $(X,Y)$ denotes simple random walk on $\bb Z^2$ jumping at total rate $2$, 
$\mc A$ denotes the set $\{(0,1), (1,1), (0,0), (1,0)\}$, and $c\in(0,+\infty)$ is some constant. Plugging this into the second term in~\eqref{estimate}, and using the reversibility of $(X,Y)$ we see that we obtain a term that is bounded from above by a constant times
\begin{equation}
\frac{1}{n^2}\sum_{z\in\mc A}\int_0^{2Tn^2}\bb E_z\big[|\Delta_\alpha f(\tfrac{X_s}{n})\Delta_\alpha f(\tfrac{Y_s}{n})|\big]\, ds\,.
\end{equation}
Since $|\Delta_\alpha f(\tfrac{x}{n})\Delta_\alpha f(\tfrac{y}{n})|$ is uniformly bounded in $x$ and $y$ we finally obtain that \eqref{estimate} is bounded by a constant, which implies condition~\textbf{i)} of Aldous' criterion via Chebychev's inequality.

We now check \textbf{ii)}. For this purpose,
fix a stopping time $\tau \in \mc T_T$. By Chebychev's inequality and repeating the argument above, we have that
 \begin{equation*}
\mathbb{P}_{\mu^n}\Big(\Big|  \int_{\tau}^{\tau+\lambda}\mc Y_s^n(\Delta_\alpha f)\, ds\;\Big| >\varepsilon\Big)
	\;\leq\; \frac{1}{\varepsilon^2} \mathbb{E}_{\mu^n}\Big[ \Big(  \int_{\tau}^{\tau+\lambda}\mc Y_s^n(\Delta_\alpha f)\, ds \;\Big)^2\Big]
	\;\leq\; \frac{\delta^2 c}{\varepsilon^2}\,,
\end{equation*}
which vanishes as $\delta\rightarrow{0}$, and yields tightness of the integral term, and concludes therefore the proof.

\subsection{Uniqueness of the  Ornstein-Uhlenbeck process}

The existence of the Ornstein-Uhlenbeck process solution of \eqref{eq Ou} is a consequence of tightness proved  in  Subsection \ref{subsec3.5}. This subsection is devoted to the proof of uniqueness of this process, as stated in Proposition~\ref{prop26}. The guideline is mainly inspired by \cite{HolleyStroock,kl}.

In the proof of Proposition~\ref{prop26} we make use of the following result, which is a standard fact about local martingales.
\begin{proposition}\label{A2}
If $M_t$ is a  local martingale with respect to a filtration $\mc F_t$ and 
\begin{align}\label{dominated}
\bb E \Big[\sup_{0 \leq s \leq t} |M_s|\,\Big]\;<\;+\infty
\end{align}
for any $t\geq 0$, then $M_t$ is a martingale.
\end{proposition}
\begin{proof}
Let $\tau_n$ be a sequence of stopping times such that $\tau_n\to\infty$ as $n\to \infty$ and such that the stopped process $(M_{t\wedge \tau_n})_{t\geq 0}$ is a martingale for each $n$. Let $s<t$, it then follows that for any $A\in \mc F_s$,
\begin{align*}
\bb E\big[ M_{t\wedge \tau_n}{\bf{1}}_A\big]\;=\; \bb E\big[ M_{s\wedge \tau_n}{\bf{1}}_A\big]\,.
\end{align*}
Letting $n\to\infty$, using \eqref{dominated} and the Dominated Convergence Theorem, we conclude that
\begin{align*}
\bb E\big[ M_{t}{\bf{1}}_A\big]\;=\; \bb E\big[ M_{s}{\bf{1}}_A\big]\,,
\end{align*}
 thus finishing the proof.
\end{proof}

\begin{proof}[Proof of Proposition~\ref{prop26}]
Fix $f\in\mc S_\alpha(\bb R)$ and $s>0$. Recall the definition of the martingales $\mc M_t(f)$ and $\mc N_t(f)$
given in \eqref{eq2.13} and \eqref{eq2.14}, respectively.

We claim that the process $\{X_t^s(f)\,:\,t\geq s\}$ defined by
\begin{equation*}
 X_t^s(f)\;=\;\exp\Bigg\{\frac{1}{2}\int_s^t \| \nabla_\alpha f\|_{\rho_r(\cdot)}^2\,dr + i\Big( \Y_t(f) - \Y_s(f) -\int_s^t\Y_r(\A
f)\,dr\Big)\Bigg\}
\end{equation*}
is a (complex) martingale. By \cite[pp. 148, Proposition 3.4]{ry} it is immediate that $X^s_t(f)$ is a local martingale. Therefore, if we show that
\begin{align}\label{eq39}
\bb E \Big[\sup_{s \leq u \leq t} |X_u^s(f)|\,\Big]\;<\;+\infty\,,
\end{align}
then, by Proposition~\ref{A2}, we conclude that $X^s_t(f)$ is a  martingale. But \eqref{eq39} is a simple consequence of the fact that the function $t\mapsto$ $\frac{1}{2}\int_0^t \| \nabla_\alpha f\|_{\rho_s(\cdot)}^2ds$ is continuous, hence bounded on compact sets. 
 Therefore, the claim is proved.

Fix $S>0$. We claim now that the process $\{Z_t\,:\, 0\leq t\leq S\}$  defined by
\begin{equation*}
 Z_t(f)\;=\;\exp\Big\{ \frac{1}{2}\int_0^t\Vert \B T^\alpha_{S-r} f\Vert^2_{\rho_r(\cdot)}\,dr +i\,\Y_t(T^\alpha_{S-t} f)\Big\}
\end{equation*}
is also a martingale.
To prove this second claim, consider two times $0\leq t_1<t_2\leq S$ and a partition of the interval $[t_1,t_2]$ in $n$
intervals of equal size, that is, $t_1=s_0<s_1<\cdots<s_n=t_2\,,$
with $s_{j+1}- s_j=(t_2-t_1)/n$. Observe now that
\begin{equation*}
\begin{split}
&\prod_{j=0}^{n-1} X_{s_{j+1}}^{s_j}(T^\alpha_{S-s_j}f)\;=\;\exp\Bigg\{  \sum_{j=0}^{n-1} \frac{1}{2}\int_{s_j}^{s_{j+1}} \| \nabla_\alpha T^\alpha_{S-s_j}f\|_{\rho_s(\cdot)}^2\,ds
\\
& +i\,\sum_{j=0}^{n-1}
\Big( \Y_{s_{j+1}}(T^\alpha_{S-s_j}f) - \Y_{s_j}(T^\alpha_{S-s_j}f) -\int_{s_j}^{s_{j+1}}\Y_r(\A
T^\alpha_{S-s_j}f)\,dr\Big)\Bigg\}\,.\\
\end{split}
\end{equation*}
Due to smoothness of $T^\alpha_t f$, the first sum in the exponential above converges to
\begin{equation*}
\frac{1}{2}\int_{t_1}^{t_2}\Vert \B T^\alpha_{S-r}f\Vert^2_{\rho_s(\cdot)} \,dr,
\end{equation*}
as $n\to +\infty$. The second sum inside the exponential is the same as
\begin{align*}
 &\Y_{t_2}(T^\alpha_{S-t_2+\frac{1}{n}}f)-\Y_{t_1}(T^\alpha_{S-t_1}f)\\
 &+ \sum_{j=1}^{n-1}
\Big( \Y_{s_{j}}(T^\alpha_{S-s_{j-1}}f-T^\alpha_{S-s_j}f) -\int_{s_j}^{s_{j+1}}\Y_r(\A
T^\alpha_{S-s_j}f)\,dr\Big).
\end{align*}
Since $\Y\in \mc C([0,T],\mc S_\alpha'(\bb R))$, since $T^\alpha_t f$ is  continuous in time and applying the expansion
$T_{t+\eps}^\alpha f-T^\alpha_t f=\eps \A T_t^\alpha f+o(\eps)$,
one can show  that the almost sure limit of the previous expression is 
$\Y_{t_2}(T^\alpha_{S-t_2}f)-\Y_{t_1}(T^\alpha_{S-t_1}f)$, see \cite{fgn2,fgn2corrigendum} for more details.
We have henceforth deduced that
\begin{align*}
 &\lim_{n\to {+\infty}} \prod_{j=0}^{n-1} X_{s_{j+1}}^{s_j}(T^\alpha_{S-s_j}f)\\
 &=\;\exp\Bigg\{
 \frac{1}{2}\int_{t_1}^{t_2}\Vert \B T^\alpha_{S-r}f\Vert^2_{\rho_s(\cdot)} \,dr+
 i\Big(\Y_{t_2}(T^\alpha_{S-t_2}f)-\Y_{t_1}(T^\alpha_{S-t_1}f)\Big)\Bigg\}\;=\; \frac{ Z_{t_2}}{Z_{t_1}}\,.
\end{align*}
Since the complex exponential is bounded, the Dominated Convergence
Theorem ensures also the
convergence in $L^1$. Thus,
\begin{equation*}
 \bb E\Big[g\,\frac{Z_{t_2}}{Z_{t_1}} \Big] \;=\;\lim_{n\to {+\infty}} \bb E\Big[g\,\prod_{j=0}^{n-1}
X_{s_{j+1}}^{s_j}(T^\alpha_{S-s_j}f) \Big]\,,
\end{equation*}
for any bounded function $g$. Take $g$ bounded and  $\mc F_{t_1}$-measurable. For any $f\in \mc
S_\alpha(\bb R)$, the process $X_t^s(f)$ is a martingale. Thus, taking the
conditional expectation with respect to
$\mc F_{s_{n-1}}$, we get
\begin{equation*}
 \bb E\Big[g\,\prod_{j=0}^{n-1}
X_{s_{j+1}}^{s_j}(T^\alpha_{S-s_j}f) \Big]\;=\;\bb E\Big[g\,\prod_{j=0}^{n-2}
X_{s_{j+1}}^{s_j}(T^\alpha_{S-s_j}f) \Big]\,.
\end{equation*}
By induction, we conclude that
\begin{equation*}
  \bb E\Big[g\,\frac{Z_{t_2}}{Z_{t_1}} \Big]\; =\; \bb E\big[\,g \,\big]\,,
\end{equation*}
for any  bounded and $\mc F_{t_1}$-measurable function $g$. This assures that $\{Z_t\,:\,t\geq 0\}$ is a martingale. From
$\bb E[Z_{t} | \mc F_s] = Z_{s}$, we get
\begin{equation*}
\begin{split}
 &\bb E\Big[\exp\Big\{ \frac{1}{2}\int_0^t\Vert \B T^\alpha_{S-r} f\Vert^2_{\rho_r(\cdot)}\,dr
+i\,\Y_t(T^\alpha_{S-t} f)\Big\}\Big\vert \mc F_s\Big]\\
&= \;\exp\Big\{ \frac{1}{2}\int_0^s\Vert \B T^\alpha_{S-r} f\Vert^2_{\rho_r(\cdot)}\,dr +i\,\Y_s(T^\alpha_{S-s} f)\Big\}\,,
\end{split}
\end{equation*}
which leads to
\begin{equation*}
\bb E\Big[\exp\Big\{ i\,\Y_t(T^\alpha_{S-t}
f)\Big\}\Big\vert \mc F_s\Big] \;=\; \exp\Big\{ -\frac{1}{2}\int_s^t\Vert \B T^\alpha_{S-r} f\Vert^2_{\rho_r(\cdot)}\,dr
+i\,\Y_s(T^\alpha_{S-s} f)\Big\}\,.
\end{equation*}
Choosing $S=t$ and replacing $f$ by $\lambda f$, we achieve
\begin{equation*}
\bb E\Big[\exp\Big\{ i\,\lambda\,\Y_t(f)\Big\}\Big\vert \mc F_s\Big] \;=\; \exp\Big\{ -\frac{\lambda^2}{2}\int_s^{t}\Vert \B
T^\alpha_{t-r} f\Vert^2_{\rho_r(\cdot)}\,dr
+i\,\lambda\,\Y_s(T^\alpha_{t-s} f)\Big\}\,,
\end{equation*}
meaning that, conditionally to $\mc F_s$, the random variable $\Y_t(f)$ has  Gaussian
distribution of mean $\Y_s(T^\alpha_{t-s}f)$ and variance $\int_s^{t}\Vert \B T^\alpha_{r} f\Vert^2_{\rho_s(\cdot)}\,dr$.

 We claim now that 
this last result implies the uniqueness of the finite dimensional distributions of  the process $\{\Y_t(f)\,:\,t\in{[0,T]}\}$. For the sake of clarity, consider only two times, $t_0=0$ and $t_1>0$, two test functions $f_0,f_1\in \mc S_\alpha(\bb R)$ and two Lebesgue measurable sets $A_0$ and $A_1$. By conditioning, 
\begin{align*}
\bb P\Big[\mc Y_{t_1}(f_1)\in A_1,\mc Y_{t_0}(f_0)\in A_0\Big]
\;=\;
\bb E\Big[\,\bb E\big[\mb 1_{[\mc Y_{t_1}(f_1)\in A_1]}\big|\mc F_{0}\big]\cdot \big[\mb 1_{[\mc Y_{t_0}(f_0)\in A_0]}\big]\Big]\,.
\end{align*}
Since the conditional expectation $\bb E\big[\mb 1_{[\mc Y_{t_1}(f_1)\in A_1]}\big|\mc F_{0}\big]$ is a function of $\mc Y_{t_0}(f_1)$ and $\mc Y_{t_0}$  is uniquely distributed as a random element of $\mc S_\alpha'(\bb R)$ (by assumption \textbf{ii)} of Proposition~\ref{prop26}), we get that the distribution of the vector $(\mc Y_{t_1}(f_1),\mc Y_{t_0}(f_0))$ is also uniquely distributed. The generalization for a general finite number of times is straightforward.

 This proves the claim, implying the uniqueness in law of the random element $\Y$ and hence finishing the proof.

\end{proof}

\subsection{Characterization of limit points}
From the results of the previous subsection we know that the sequence $\{\mc Y_t^n:t\in [0,T]\}_{n\in \bb N}$ has limit points.
Let $\{\mc Y_t:t\in [0,T]\}$ be the limit in distribution of   $\{\mc Y_t^n:t\in [0,T]\}_{n\in \bb N}$ along some subsequence $n_k$ considering the uniform topology of $\mathcal{D}([0,T],\mathcal{S}'_{\alpha}(\bb R))$. Abusing of notation, we denote this subsequence simply by $n$. Our goal here is to prove that $\{\mc Y_t:t\in [0,T]\}$ satisfies  the conditions \textbf{i)} and \textbf{ii)} of Proposition~\ref{prop26}. Since Proposition~\ref{convergence at time zero} gives us  condition \textbf{ii)}, it only remains  to prove condition \textbf{i)}.

For  $f\in \mc S_\alpha(\bb R)$, let $\mc M_t$ and $\mc N_t$ be the processes \textit{defined} by
\begin{align*}
\mc M_t(f)&\;=\; \mc Y_t(f) -\mc Y_0(f) -  \int_0^t \mc Y_s(\A f)ds\,,\\
\mc N_t(f)&\;=\;\big(\mc M_t(f)\big)^2 - \int_0^t \| \nabla_\alpha f\|_{\rho_s(\cdot)}^2\,ds\,.
\end{align*}
 Since $\mc Y_t^n$  is assumed to converge in distribution to $\mc Y_t$ as $n\to+\infty$, by \eqref{martdecomp3.4} and\break Lemma~\ref{lemma31aa}, we conclude that $\mc M_t(f)$ defined above coincides with the limit of $\mc M^n_t(f)$ as in Lemma~\ref{lemma32}, which was denoted by $\mc M_t(f)$ as well.

By  Lemma~\ref{lemma32}, we already know that  $\mc M_t(f)$ has  quadratic variation given by $\int_0^t \| \nabla_\alpha f\|_{\rho_s(\cdot)}^2\,ds$. Therefore, if we show that $\mc M_t(f)$ is a martingale, then we will  immediately get that $\mc N_t(f)$ is also a martingale.

Hence, we claim that $\mc M_t(f)$ is a martingale. First of all, we fix the filtration, which will be the natural one: $\mc F_t=\{\sigma(\mc Y_s(g)):s\leq t \text{ and } g\in \mc S_\alpha(\bb R)\}$. Thus, $\mc M_t(f)$ is $\mc F_t$-measurable.
The fact that $\mc M_t(f)$ is in $L^1$ for any time $t\in[0,T]$ is a consequence that $\mc M_t(f)$ is a Gaussian process, which was proved in Lemma~\ref{lemma32}. Thus, if we prove that 
\begin{equation}\label{eq315}
\bb E\big[\mc M_t(f) {\bf 1}_U\big]\;=\; \bb E\big[\mc M_s(f) {\bf 1}_U\big]\,,\quad\forall\, U \in \mc F_s\,,
\end{equation}
we will conclude that $\mc M_t(f)$ is a martingale. To assure \eqref{eq315} it is enough to verify it for sets $U$ of the form 
\begin{align*}
U\;=\; \bigcap_{i=1}^k\big[\mc Y_{s_i}(f_i)\in A_i\big]
\end{align*}
for $0\leq s_1\leq \cdots \leq s_k\leq s$, $f_i\in \mc S_\alpha(\bb R)$ and $A_i$ measurable sets of $\bb R$.
Since $\mc M^n_t(f)$ is a martingale, 
\begin{equation}\label{eq316}
\bb E\big[\mc M_t^n(f) {\bf 1}_{U_n}\big]\;=\; \bb E\big[\mc M_s^n(f) {\bf 1}_{U_n}\big]\,,\quad\forall\, U \in \mc F_s\,,
\end{equation}
where 
\begin{align*}
U_n\;=\; \bigcap_{i=1}^k\big[\mc Y_{s_i}^n(f_i)\in A_i\big]
\end{align*}
for $0\leq s_1\leq \cdots \leq s_k\leq s$, $f_i\in \mc S_\alpha(\bb R)$ and $A_i$ are measurable sets of $\bb R$. Therefore, in order to show \eqref{eq315} it is enough to prove the claim that the expectations in \eqref{eq316} converge to the respective expectations in \eqref{eq315}. 

Since $\mc Y_t^n(f)$ converges to $\mc Y_t(f)$ as $n\to+\infty$, which is concentrated on continuous paths,  then $\mc M_t^n(f) {\bf 1}_{U_n}$ converges in distribution to $\mc M_t(f) {\bf 1}_{U}$. Thus, by
\cite[pp 32, Theorem 5.4]{Bill} in order to get convergence of expectations, it is enough to assure that $\{\mc M_t^n(f) {\bf 1}_{U_n}\}_{n\in\bb N}$ is a uniformly integrable sequence. In its hand, the uniform integrability can be guaranteed by showing that the $L^2$ norm of $\mc M_t^n(f) {\bf 1}_{U_n}$ is uniformly bounded in $n\in\bb N$. Since the indicator function  is bounded by one, we can deal only with the $L^2$ norm of the martingale $\mc M_t^n(f)$. Now, applying the Minkowksi inequality to \eqref{eq:mart_decom}, we get
\begin{equation}\label{eq319}
\begin{split}
\bb E_{\mu_n}\big[\big(\mc M_t^n(f)\big)^2 \big]^{1/2} \;\leq\;& \bb E_{\mu_n}\big[\big(\mc Y_t^n(f)\big)^2 \big]^{1/2}+
\bb E_{\mu_n}\big[\big(\mc Y_0^n(f)\big)^2 \big]^{1/2}\\
& +\bb E_{\mu_n}\Big[\Big(\int_{0}^{t}\pfrac{1}{\sqrt{n}}\sum_{x\in{\bb Z}}n^2\mca_n f\big(\pfrac{x}{n}\big)\overline{\eta}_s(x)\,ds\Big)^2 \Big]^{1/2}\,.
\end{split}
\end{equation}
The first term on the right hand-side of \eqref{eq319} is bounded by
\begin{align*}
\frac{1}{{n}}\sum_{x\in \bb Z}\big(f(\tfrac{x}{n})\Big)^2\chi(\rho^n_t(x))
+\frac{2}{{n}}\sum_{x<y} f(\tfrac{x}{n}) f(\tfrac{y}{n})\varphi^n_t(x,y)\,.
\end{align*}
Since $|\rho_t^n(x)|\leq 1$, the first parcel in the display above is uniformly bounded in $n$. To treat the second term of the last display, we use a similar argument to the one used below \eqref{estimate}.  The second term on the RHS of \eqref{eq319} is bounded by
\begin{align*}
\frac{1}{{n}}\sum_{x\in \bb Z}\big(f(\tfrac{x}{n})\big)^2\chi(\rho^n_0(x))
+\frac{2}{{n}}\sum_{x<y} f(\tfrac{x}{n}) f(\tfrac{y}{n})\varphi^n_0(x,y)\,,
\end{align*}
which is uniformly bounded on $n\in \bb N$ due to conditions \eqref{assumptionone} and \eqref{ass2}.
Again by a similar argument to the one presented for tightness below \eqref{estimate}, the third term on the  right hand-side of \eqref{eq319} is  bounded by $t^2$ times
\begin{equation*}
\frac{1}{{n}}\sum_{x\in \bb Z}\big( f(\tfrac{x}{n})\big)^2\sup_{t\leq{T}}\chi(\rho^n_t(x))
+\frac{2}{{n}}\sum_{x<y} f(\tfrac{x}{n}) f(\tfrac{y}{n})\sup_{t\leq{T}}\varphi^n_t(x,y)\,,
\end{equation*}
thus concluding the characterization of limit points.

\appendix

\section{Auxiliary results on random walks}\label{A1}
The next result is quite classical, but hard to find in the literature. It is  included  here for sake of completeness. 
\begin{proposition}\label{prop:1dwalk}
Let ${X}$ be the symmetric simple one-dimensional continuous time random walk. Then,
\begin{align*}
\int_0^t \bb P\big[\, X_s=0\,\big]\,ds\; \leq\; c\sqrt{t}\,,
\end{align*}
where $c>0$ is a constant which does not depend on $t$.
\end{proposition}
\begin{proof}
Let $N:=N_{2s}$ a Poisson distribution with parameter $2s$.
\begin{equation} \label{rw}
\begin{split}
\bb P\big[ X_s=0\big]&=\sum_{k=0}^\infty \bb P\big[ X_k=0 |N=k\big]\cdot \bb P\big[ N=k\big] \\
&=\sum_{k=0}^\infty \textbf{1}_{[k \, \textrm{is even}]}\frac{1}{2^k}\left(
\begin{array}{c}
k\\
k/2
\end{array}
\right)\bb P[N=k]\\
&=e^{-2s}+\sum_{k=1}^{\left\lfloor s\right\rfloor} \textbf{1}_{[k \, \textrm{is even}]}\frac{1}{2^k}\left(
\begin{array}{c}
k\\
k/2
\end{array}
\right)\bb P[N=k]\\
&+\sum_{k=\left\lfloor s\right\rfloor+1}^\infty \textbf{1}_{[k \, \textrm{is even}]}\frac{1}{2^k}\left(
\begin{array}{c}
k\\
k/2
\end{array}
\right)\bb P[N=k]\,.
\end{split}
\end{equation}
Using the Stirling Formula (seefor example Feller, Vol I.), it is easy to check that 
\begin{equation}\label{stirling}
\frac{1}{2^k}\binom{k}{k/2}\;\leq\;  \frac{1}{\sqrt{\pi k}} \;\leq\; 1\,.
\end{equation}
Applying the second inequality of \eqref{stirling} in the first sum of \eqref{rw} and the first inequality of \eqref{stirling} in the second sum in \eqref{rw}, we obtain that $\bb P\big[ X_s=0\big]$ is  bounded from above by
\begin{align} \label{rw1}
& e^{-2s}+ \bb P\big[N\leq \left\lfloor s\right\rfloor\big]+ \frac{c_1}{\sqrt{s}}\sum_{k=\left\lfloor s\right\rfloor+1}^\infty \bb P[N=k]\; \leq\;  e^{-2s}+ \bb P\big[N\leq \left\lfloor s\right\rfloor\big]+ \frac{c_1}{\sqrt{s}}\,.
\end{align}
In the sequel, we will get an exponential bound  $\bb P\big[\,N\leq \left\lfloor s\right\rfloor\,\big]$ by a standard large deviations technique. In this way, note that, for any  $\theta>0$,
\begin{align*}
\bb P\big[\,N\leq \left\lfloor s\right\rfloor\,\big]& \;=\;\bb E\big[\,\textbf{1}_{[N\leq s]}e^{\theta N}e^{-\theta N}\,\big]\;\leq\; e^{\theta s}\,\bb E\big[\,\textbf{1}_{[N\leq s]}e^{-\theta N}\,\big]\\
&\;\leq\; e^{\theta s} \,\bb E\big[\,e^{-\theta N}\,\big]\;=\;e^{\theta s}e^{2s(e^{-\theta} -1)}\;=\; e^{s(2e^{-\theta} -2+ \theta)}\,.
\end{align*}
Denote $f(\theta)=2e^{-\theta} -2 + \theta$ and note that  $f$ assumes its minimum at $\theta_0=\log 2>0$, and $f(\theta_0)=\log 2-1<0$. Therefore, choosing  $\theta=\theta_0$, we get 
$$\bb P\big[\,N\leq\left\lfloor s\right\rfloor\,\big] \;\leq\; e^{s(\log 2 -1)}\,.$$
Looking at \eqref{rw1} and then to \eqref{rw}, we conclude that 
$$\bb P\big[\, X_s=0\,\big]\;\leq\; e^{-2s}+e^{s(\log 2 -1)} + \frac{c_1}{\sqrt{s}}\,.$$
Integrating, we get
\begin{equation*}
\int_0^t \bb P\big[\, X_s=0\,\big]\,ds \;\leq\; \int_0^t \Big(e^{-2s}+e^{s(\log 2 -1)} + \frac{c_1}{\sqrt{s}}\Big)\,ds \;\leq\; c_2\sqrt{t}\,,
\end{equation*}
for some constant $c_2$  not depending on $t$. 

\end{proof}

\begin{proposition}\label{prop:2dwalk} 
Let ${(X,Y)}$ be the symmetric simple two-dimensional continuous time random walk. Then,
\begin{align*}
\int_0^t \bb P\big[\, (X_s,Y_s)=(0,0)\,\big]\,ds\; \leq\; c\log{t}\,,
\end{align*}
where $c>0$ is a constant which does not depend on $t$.
\end{proposition}
The proof of the statement above can be  adapted from the one of Proposition~\ref{prop:1dwalk}.

\section{Fluctuations  at the initial time}
\begin{proposition} \label{convergence at time zero}Let 
$\nu^n_{\rho_0(\cdot)}$ be the slowly varying Bernoulli product measure  associated with a smooth profile $\rho_0$.
Then, $\mc Y^n_0$ converges in distribution to $\mc Y_0$,
where $\mc Y_0$
 is a mean zero Gaussian field  of  covariance given by \begin{equation}\label{eq:covar1}
\mathbb{E}\Big[ \mc Y_0(g) \mc Y_0(f)\Big] \;=\;  \int_{\mathbb{R}} \chi\big(\rho_0(u)\big)\, g(u) \,f(u)\,  du\,,
\end{equation}
for any $f,g\in\mc S_\alpha(\bb R)$.
\end{proposition}
\begin{proof} As argued  in Subsection~\ref{subsec3.5}, for each $f\in \mc S_\alpha(\bb R)$, the sequence\break $\big\{\mc Y_0(f)\big\}_{n\in \bb N}$ is tight, hence $\big\{\mc Y_0\big\}_{n\in \bb N}$ is  tight due to Mitoma's criterion (Proposition~\ref{mitoma}). Thus, it remains only to characterize the joint limit in distribution for the vectors of the form $\big(\mc Y_0(f_1),\ldots,\mc Y_0(f_k)\big)$, with $f_i\in\mc S_\alpha(\bb R)$, for $i=1,\cdots,k$.
 Since $\nu_{\rho_0}^n$ is a  product measure,
\begin{equation*}
\begin{split}
&\log\mathbb{E}_{\nu_{\rho_0}^n(\cdot)}\Big[\exp\Big\{i\theta\mathcal{Y}^n_{0}(f)\Big\}\Big]
\;=\;\sum_{x\in{\bb Z}}\log\mathbb{E}_{\nu_{\rho_0}^n(\cdot)}\Big[\exp\Big\{\frac{i\theta}{\sqrt{n}}\;
\bar{\eta}_0(x)f\Big(\frac{x}{n}\Big)\Big\}\Big]\\
& =\! \sum_{x\in \bb Z}\log \Big[\rho_0(\pfrac{x}{n})\exp\big\{\pfrac{i\theta}{\sqrt{n}}f(\pfrac{x}{n})\big(1-\rho_0(\pfrac{x}{n})\big)\big\}+\big(1-\rho_0(\pfrac{x}{n})\big)\exp\big\{-\pfrac{i\theta}{\sqrt{n}}f(\pfrac{x}{n})\rho_0(\pfrac{x}{n})\big\}\Big]\!.
\end{split}
\end{equation*}
Since $f\in \mc S_\alpha(\bb R)$, we have smoothness of $f$ except possibly at $x=0$, together with fast decaying. Keeping this in mind,   Taylor's expansion on the exponential function permits  to conclude that the expression above is equal to
\begin{equation*}
-\frac{\theta^2}{2n}\sum_{x\in{\bb Z }}f^{2}\Big(\frac{x}{n}\Big)\chi(\rho_0(\pfrac{x}{n}))+O(\tfrac {1}{\sqrt n})\,,
\end{equation*}
which gives us that
\begin{equation*}\label{eq23aaa}
\lim_{n\rightarrow{{+\infty}}}\log\mathbb{E}_{\nu_{\rho_0}^n(\cdot)}\Big[\exp\Big\{i\theta{\mathcal{Y}^n_{0}(f)}\Big\}\Big]\;=\;-\frac{\theta^2}{2}
\int_{\mathbb{R}}\chi\big(\rho_0(u)\big)\, f^2(u)\,du\,.
\end{equation*}
Replacing $f$ by a
linear combination of functions and then applying the\break Cr\'amer-Wold device, the proof ends.
\end{proof}

\section*{Acknowledgements}

A. N. thanks ``L'OR\' EAL - ABC - UNESCO Para Mulheres na Ci\^encia''. T.F.  and P.G. would like to thank the hospitality of the  Center of Mathematics of the University of Minho in Portugal, where this work was initiated.  This project has received funding from the European Research Council (ERC) under  the European Union's Horizon 2020 research and innovative programme (grant agreement   No 715734). T.F. was supported through a project Jovem Cientista-9922/2015, FAPESB-Brazil. M.T. would like to thank CAPES for a PDSE scholarship, which supported her studies when visiting P.G. in IST, Portugal.

\bibliographystyle{plain}
\bibliography{bibliography}
\end{document}